\documentclass[11pt,a4paper]{article}
\nonstopmode
\def\version{December 8, 2020}
\usepackage{latexsym,amssymb,bm,amsmath}
\usepackage[dvips]{graphicx}

\usepackage{amsthm}
\usepackage[normalem]{ulem}

\usepackage{color}
\usepackage[backref=none]{hyperref}
\definecolor{MyDarkBlue}{rgb}{0,0.08,0.45}
\definecolor{MyDarkGreen}{rgb}{0,0.75,0}
\definecolor{Pomegranade}{rgb}{0.6,0.1,0.15}
\definecolor{purple}{rgb}{0.6,0.1,0.15}
\usepackage[backref=none]{hyperref}
\hypersetup{pdfborder={0 0 0},
  colorlinks,
  urlcolor={MyDarkBlue},
  linkcolor={MyDarkBlue},
  citecolor={MyDarkBlue},
  breaklinks=true}

\newcommand{\Span}{\mathop{\mathrm{Span}}}

\newcommand{\beqn}{\begin{eqnarray}}
\newcommand{\eeqn}{\end{eqnarray}}

\newcommand{\be}{\begin{equation}}
\newcommand{\ee}{\end{equation}}
\newcommand{\ba}{\begin{array}}
\newcommand{\ea}{\end{array}}

\newcommand{\p}{\partial}

\renewcommand{\Im}{\,\mathrm{Im}\,}
\renewcommand{\Re}{\,\mathrm{Re}\,}

\newcommand{\fra}[2]{{#1}/{#2}}

\newcommand{\cH}{{\cal H}}

\newcommand{\calC}{\mathcal{C}}

\newcommand{\cS}{{\cal S}}

\newcommand{\calE}{\mathcal{E}}
\newcommand{\calQ}{\mathcal{Q}}
\newcommand{\calT}{\mathcal{T}}
\newcommand{\calX}{\mathcal{X}}

\newcommand\C{{\mathbb C}}
\newcommand\R{{\mathbb R}}

\newcommand{\bfA}{\mathbf{A}}

\newcommand{\bfG}{\mathbf{G}}
\newcommand{\bfJ}{\mathbf{J}}
\newcommand{\bfH}{\mathbf{H}}

\newcommand{\bmSigma}{\bm\Sigma}
\newcommand{\dom}{\mathfrak{D}}

\newcommand{\jj}{\mathrm{i}}
\newcommand{\e}{\bm{e}}
\newcommand{\w}{\bm{w}}

\newtheorem{theorem}{Theorem}[section]

\newtheorem{lemma}[theorem]{Lemma}

\newtheorem{corollary}[theorem]{Corollary}

\theoremstyle{definition}

\newtheorem{definition}[theorem]{Definition}
\newtheorem{remark}[theorem]{Remark}

\makeatletter\@addtoreset{equation}{section}
\makeatother
\providecommand{\itref}[1]{{\it (\ref{#1})}}

\newcommand{\bd}{\begin{definition}}
\newcommand{\ed}{\end{definition}}
\newcommand{\bt}{\begin{theorem}}
\newcommand{\et}{\end{theorem}}

\newcommand{\bp}{\begin{pro}}
\newcommand{\ep}{\end{pro}}

\newcommand{\bl}{\begin{lemma}}
\newcommand{\el}{\end{lemma}}
\newcommand{\bc}{\begin{corollary}}
\newcommand{\ec}{\end{corollary}}

\newcommand{\br}{\begin{remark} }
\newcommand{\er}{\end{remark}}
\newcommand{\brs}{\begin{remarks}}
\newcommand{\ers}{\end{remarks}}

\newcommand{\abs}[1]{\vert #1\vert}

\newcommand{\norm}[1]{\Vert #1\Vert}

\begin{document}

%%%%%%%%%%%%%%%%%%%%%%%%%%%%%%%%%%%%%%%%%%%%%%%%%%%%%%%%%%%%%%%%%%%%%%%%%
%%BEGINNING OF TEXT
%%%%%%%%%%%%%%%%%%%%%%%%%%%%%%%%%%%%%%%%%%%%%%%%%%%%%%%%%%%%%%%%%%%%%%%%%

\title{
On spectral and orbital stability 
for the Klein--Gordon equation coupled
to an anharmonic oscillator}

\author{
{\sc Andrew Comech}
\\
{\small\it Texas A\&M University, College Station, TX}
\\
{\small\it Institute for Information Transmission Problems, Moscow, Russia}
\\ \\
{\sc Elena A. Kopylova}
\\
{\small\it Institute for Information Transmission Problems, Moscow, Russia}
\\
{\small\it Vienna University, Vienna, Austria}
}

\date{\version}

\maketitle

\begin{abstract}
We obtain explicit characterization of spectral and orbital stability
of solitary wave solutions to the
$\mathbf{U}(1)$-invariant
Klein--Gordon equation
in one spatial dimension
coupled to an anharmonic oscillator.
We also give the complete analysis of the spectrum
of the linearization at a solitary wave.
\end{abstract}

%%%%%%%%%%%%%%%%%%%%%%%%%%%%%%%%%%%%%%%%%%%%%%%%%%%%%%%%%%%%%%%%%%%%%%%%%%%%%%%%%%
%%%%%%%%%%%%%%%%%%%%%%%%%%%%%%%%%%%%%%%%%%%%%%%%%%%%%%%%%%%%%%%%%%%%%%%%%%%%%%%%%
\section{Introduction}
\label{intr}
%%%%%%%%%%%%%%%%%%%%%%%%%%%%%%%%%%%%%%%%%%%%%%%%%%%%%%%%%%%%%%%%%%%%%%%%%%%%%%%%%%%%

In the present article we study the spectral stability
of the $\mathbf{U}(1)$-invariant
Klein--Gordon equation on a line with a concentrated nonlinearity:
\begin{equation}\label{KG}
\ddot\psi(x,t)=
\p_x^2\psi(x,t)-m^2\psi(x,t)+\delta(x)a(\abs{\psi(0,t)}^2)\psi(0,t),\qquad
\psi(x,t)\in\C,
\quad
x\in\R,
\end{equation}
where $m>0$.
Above, $a(\cdot)$ is a real-valued
differentiable function,
so that the model is $\mathbf{U}(1)$-invariant.
The equation is understood in the sense of distributions.
Physically, equation  (\ref{KG}) describes  the Klein--Gordon field
coupled to a nonlinear oscillator located at $x=0$,
with
$a(\abs{\psi}^2)\psi$ being the oscillator force.
Equation (\ref{KG})
admits finite energy  solutions of the form
$\phi(x,\omega)e^{-\jj\omega t}$, $\omega\in\R$,
called  \emph{solitary waves}.
The  solitary waves form a two-dimensional \emph{solitary manifold}
in the Hilbert space of finite energy states of the system.

We will
determine conditions for orbital and spectral stability of solitary waves,
formulating the results in terms of values of $a$ and $a'$ evaluated
at $\abs{\phi(0,\omega)}^2$;
see Theorem~\ref{theorem-stability} below.
Then, in Theorem~\ref{theorem-d-lambda},
we give the complete description of the spectrum of the
linearization at a solitary wave.

Equation \eqref{KG}
was proposed to model electron's transitions
between Bohr's \emph{quantum orbits},
which is one of the fundamental problems of Quantum Mechanics.
More precisely,
\eqref{KG} models the interaction of the electron
(represented by the anharmonic oscillator)
with the electromagnetic
radiation (represented by the Klein--Gordon field).
The global attraction of any finite energy solution
to the set of all
solitary waves in local energy norms,
which was established in \cite{MR2308860},
can be interpreted as the
relaxation of a perturbed electron
to the Bohr orbit, where it no longer loses the energy
via the radiation,
in the complete agreement with Bohr's postulate
on quantum jumps.

%% The solitary waves were introduced by Schr\"odinger for the quantum electron
%% coupled to the Maxwell field~\cite{sch-386-109}.
%% He identified the Bohr orbits,
%% or quantum stationary states,
%% with the quasiperiodic solutions
%% of the Schr\"odinger equation.

We pursue further properties of the model \eqref{KG}
since it is a convenient playground
for establishing asymptotic stability results,
similarly to
\cite{buslaev2008asymptotic,komech2012asymptotic},
allowing one to explicitly check
all the spectral properties of the linearized equation.
In the present article,
we will obtain the spectral and orbital stability results.
Moreover, we obtain the complete description of the spectrum
of the linearized equation:
these results will be needed for the subsequent proof
of asymptotic stability.
We point out that many pieces of
the spectral analysis which we develop in the present article
can be carried over essentially verbatim to models similar to \eqref{KG},
such as the models where the concentrated nonlinearity
is substituted by its regularized versions,
such as the self-interaction based on the mean field \cite{MR2526405}.

We mention that the local and global well-posedness
of \eqref{KG} has already been proved in \cite{MR2308860}.
In \cite{kopylova2009asymptotic,kopylova2010asymptotic},
asymptotic stability of solitary waves was obtained for discrete 
Schr\"odinger and Klein-Gordon equations.
Let us also mention that related results
on local well-posedness,
orbital stability, and linear instability
of solitary waves in the nonlinear Klein--Gordon equation
in the external $\delta$-function potential
were obtained in \cite{csobo2019stability}.  

The paper is organized as follows.
The model and its solitary wave solutions are described
in Section~\ref{ndsec}.
% some notation and definitions are given.
The linearization at a solitary wave is
carried out in Section~\ref{section-3},
where the standard properties of the linearized operator
are obtained.
The detailed structure of the spectrum
of the linearized operator
is derived in
Section~\ref{section-4}
(see Theorem~\ref{theorem-d-lambda}).

\section{The model}
\label{ndsec}

We define
$\varPsi(x,t)
=
\begin{bmatrix}
\psi(x,t)\\\pi(x,t)
\end{bmatrix}
\in\C^2
$
and rewrite (\ref{KG}) in the vector form:
\begin{equation}\label{KGE1}
\dot\varPsi(t)
=\begin{bmatrix}
0&1\\
\p_x^2-m^2&0
\end{bmatrix}
\varPsi(t)+\delta(x)\begin{bmatrix}0\\a(\abs{\psi}^2)\psi\end{bmatrix},
\qquad
\varPsi=\begin{bmatrix}
\psi\\\pi
\end{bmatrix}\in\C^2,
\qquad
\varPsi\big|_{t=0}=\varPsi_0:=\begin{bmatrix}\psi_0\\\pi_0
\end{bmatrix}.
%%\in H^1(\R,\C)\times L^2(\R,\C).
\end{equation}
%We identify a complex number $\psi=\psi_1+\jj\psi_2\in\C$
%with the real two-dimensional
%vector $(\psi_1,\psi_2)\in\R^2$ and assume that the vector version
%of
The oscillator force $a(\abs{\psi}^2)\psi$
admits a real-valued potential,
\begin{equation}\label{P}
  a(\abs{\psi}^2)\psi=-\nabla_{\Re\psi,\Im\psi} U(\psi),
\qquad\psi\in\R^2,
\qquad
U\in C^2(\R^2),
\end{equation}
where $U(\psi)=u(\abs{\psi}^2)$, with
$
u(\tau)=\frac 1 2\int_0^\tau a(s)\,ds$.
%the gradient is taken with respect to $\Re\psi$ and $\Im\psi$.
Then (\ref{KGE1}) is formally a Hamiltonian system
with the Hamiltonian functional
\begin{equation}\label{H}
 {\cH}(\varPsi)=\frac 12
\int_{\R}\Big(|\pi|^2+|\p_x\psi|^2+m^2|\psi|^2\Big) \,dx+U(\psi(0)),
\qquad\varPsi =\begin{bmatrix}\psi(x)\\\pi(x)\end{bmatrix}
\in H^1(\R)\times L^2(\R),
\end{equation}
which is conserved for sufficiently regular finite energy solutions.
Equation (\ref{KG}) is $\mathbf{U}(1)$-invariant:
if $\psi(x,t)$ is a solution, then so is $e^{\jj\theta}\psi(x,t)$
for any $\theta\in\R$.
The N\"other theorem implies the \emph{charge} conservation:
the value of the functional
\begin{equation}\label{Q}
\calQ(\varPsi)
=-\frac 1 2\Im\int_{\R}
\big(\overline\psi(x)\pi(x)
-\overline\pi(x)\psi(x)\big)\,dx,
\qquad
\varPsi=\begin{bmatrix}
\psi(x)\\\pi(x)
\end{bmatrix},
\end{equation}
is conserved for solutions to \eqref{KG}.

The local and global existence result for the Cauchy problem (\ref{KGE1})
proved in~\cite[Theorem 2.1]{MR2308860}:
\begin{theorem}\label{locex}
Assume that the potential is represented by
$U(\psi)=u(\abs{\psi}^2)$ with $u\in C^2(\R)$
and that $U$ satisfies the inequality
\[
U(\psi)\ge A-B|\psi|^2\quad\mbox{for}\quad\psi\in\C,
\qquad\mbox{where}\quad A\in\R,\quad 0\le B<m.
\]
Then:
\begin{enumerate}
\item
   For every $\varPsi_0\in {\calE}:=H^1(\R)\oplus L^2(\R)$
the Cauchy problem (\ref{KGE1})
   has a unique solution
   $\varPsi\in C_b(\R,{\calE})$.
\item
The energy of $\varPsi$ is conserved:
for all $t\in\R$,
   $\cH(\varPsi(t))=\cH(\varPsi(0))$.
\item
   There exists $\varLambda(\varPsi_0)>0$ such that
$
\sup\limits\sb{t\in\R}\Vert{\varPsi(t)}\Vert\sb{\calE}
\le\varLambda(\varPsi_0)<\infty$.
\end{enumerate}
\end{theorem}

The main subject of this paper is the analysis of the spectral and orbital
stability of
``quantum stationary states'', or \emph{solitary waves}~\cite{MR901236}, 
which are  finite energy solutions of the form
\begin{equation}\label{SW}
\varPsi(x,t)=\varPhi(x)e^{-\jj\omega t},\qquad\omega\in\R,
\qquad
\varPhi(x)
=\begin{bmatrix}\phi(x)\\ -\jj\omega\phi(x)\end{bmatrix},
\qquad\phi\in H^1(\R).
\end{equation}
Above,
$H^1(\R)$ is the Sobolev space of complex-valued measurable functions
satisfying $\int_{\R} (|\p_x\psi|^2+|\psi|^2\bigr)\,dx<\infty$.
By (\ref{KGE1}),
the frequency  $\omega\in\R$
and the amplitude $\phi(x)$ solve the following
\emph{nonlinear eigenvalue problem}:
\begin{equation}\label{NEP}
\omega^2\phi(x)=-\p_x^2\phi(x)
+m^2\phi(x)-\delta(x)a(\abs{\phi(0)}^2)\phi(0),\qquad x\in\R.
\end{equation}

\begin{definition}\label{dSS}
The solitary manifold $\cS$ is the set  of all amplitudes
$\varPhi(\cdot,\omega)
=\begin{bmatrix}\phi(\cdot,\omega)\\ -\jj\omega\phi(\cdot,\omega)\end{bmatrix}$,
where $\phi(x,\omega)e^{-\jj\omega t}$ is a nonzero solitary wave solution to \eqref{KG},
with all possible $\omega\in\R$.
\end{definition}

\begin{lemma}\label{lemma-sw}
The set $\cS$ is given by
$$
\cS=\Bigl\{
\varPhi(x,\omega)e^{\jj\theta}=
\begin{bmatrix}
\phi(x,\omega)\\ -\jj\omega\phi(x,\omega)
\end{bmatrix}e^{\jj\theta},
\qquad
\phi(x,\omega)=C e^{-\varkappa|x|},\qquad\theta\in[0,2\pi]\Bigr\},
$$
where $|\omega|<m$,
$\varkappa=\sqrt{m^2-\omega^2}>0$,
and $C>0$ satisfies the relation
\begin{eqnarray}\label{a-kappa}
a(C^2)=2\varkappa.
\end{eqnarray}
If $a'(C^2)\ne 0$, then $C$ is locally a $C^1$-function of $\omega$.
\end{lemma}

\begin{proof}
Equation (\ref{NEP}) implies that
 $\p_x^2\phi(x,\omega)=(m^2-\omega^2)\phi(x,\omega)$, $ x\ne 0$,
and hence  $\phi(x,\omega)=C_\pm e^{-\varkappa_\pm\abs{x}}$ for $\pm x>0$, where $\varkappa_\pm$
satisfy $\varkappa_\pm^2=m^2-\omega^2$.
Since $\phi(\cdot,\omega)\in H^1(\R)$,
we conclude that $\varkappa_\pm>0$,
$|\omega|<m$, and that $\varkappa_\pm=\sqrt{m^2-\omega^2}>0$.
Since the function
$\phi(x,\omega)$ is continuous at $x=0$,
one has
$C_-=C_+=C$.
Thus, the solitary waves are solutions of the form
\begin{equation}\label{swa-0}
\phi(x,\omega)=C e^{-\varkappa |x|},\qquad\varkappa=\sqrt{m^2-\omega^2}> 0,\qquad C\not=0.
\end{equation}
For our convenience,
we assume that $C>0$,
and write solitary wave solutions in the form
\begin{equation}\label{swa}
\phi(x,\omega)
=
e^{\jj\theta}
C e^{-\varkappa |x|},\qquad\varkappa=\sqrt{m^2-\omega^2}> 0,
\qquad C>0,
\quad
\theta\in\R\mod 2\pi.
\end{equation}
The algebraic equation satisfied by the constant $C>0$
is obtained by collecting coefficients
at $\delta(x)$ in (\ref{NEP}):
\begin{equation}\label{ais}
  0=\p_x\phi(0+,\omega)-\p_x\phi(0-,\omega)+a(\abs{\phi(0,\omega)}^2)\phi(0,\omega).
\end{equation}
This implies that $0=-2\varkappa C+a(C^2)C$,
hence $2\varkappa=a(C^2)$.

We note that there is the following relation:
\begin{eqnarray}\label{o-h}
-2\omega
=\frac{d}{d\omega}\varkappa^2
=\frac{d}{d\omega}\frac{a(C^2)^2}{4}
=a(C^2)a'(C^2)C\frac{d C}{d\omega},
\end{eqnarray}
which shows that $C$ is locally a $C^1$-function of $\omega$
as long as $a'(C^2)\ne 0$.
\end{proof}

\section{Linearization at a solitary wave}
\label{section-3}

Let us linearize the nonlinear
Klein--Gordon (\ref{KGE1}) at a solitary wave
\[
e^{-\jj\omega t+\jj\theta}
\varPhi(x,\omega)
=
e^{-\jj\omega t+\jj\theta}
\begin{bmatrix}\phi(x,\omega)\\-\jj\omega\phi(x,\omega)
\end{bmatrix},
\]
where
$\omega\in(-m,m)$, $\theta\in\R\mod 2\pi$,
and $\phi(x,\omega)=C(\omega)e^{-\varkappa|x|}$,
with $\varkappa=\sqrt{m^2-\omega^2}> 0$ and $C(\omega)>0$
(see Lemma~\ref{lemma-sw}).
Substituting
\begin{equation}\label{prol}
\varPsi(x,t)=e^{-\jj\omega t+\jj\theta}(\varPhi(x,\omega)+\calX(x,t)),
\qquad
\calX(x,t)=\begin{bmatrix}\calX_1(x,t)\\\calX_2(x,t)\end{bmatrix}\in\C^2,
\end{equation}
into (\ref{KGE1}), we obtain:
$$
-\jj\omega(\varPhi+\calX)+\dot\calX=
\begin{bmatrix} 0&1\\\p_x^2-m^2 & 0\end{bmatrix}(\varPhi+\calX)+\delta(x)
\begin{bmatrix} 0\\
%% F(C+\calX_1)
a(\abs{C+\calX_1}^2)(C+\calX_1)
\end{bmatrix}.
$$
Equality $\varPhi_2=-\jj \omega\varPhi_1$
and equation (\ref{NEP}) lead to
\begin{equation}\label{lin1}
\dot\calX=\begin{bmatrix}\jj\omega & 1\\\p_x^2-m^2 &\jj\omega\end{bmatrix}
\calX(x,t)+\delta(x)
\begin{bmatrix}0\\
a(\abs{C+\calX_1}^2)(C+\calX_1)
-a(\abs{C}^2)C
\end{bmatrix}.
\end{equation}
%%By (\ref{I}) 
The first order part of (\ref{lin1}) is given by
\begin{equation}\label{lin3}
\dot\calX(x,t)=
\begin{bmatrix}\jj\omega & 1\\\p_x^2-m^2 &\jj\omega\end{bmatrix}
 \calX(x,t)+\delta(x)\begin{bmatrix} 0\\
a(C^2)\calX_1(0,t)+C^2 a'(C^2) 2\Re \calX_1(0,t)\end{bmatrix}.
\end{equation}
In the case $a'(C^2)\not=0$, the operator in the right-hand side is
$\R$-linear but not $\C$-linear;
to study its spectrum,
it is convenient to rewrite (\ref{lin3}) in the real form.
For $\varPsi=\begin{bmatrix}\psi\\ \dot\psi
\end{bmatrix}\in\C^2$,
we denote
$\Psi=\begin{bmatrix}
\Re\psi\\\Im\psi\\\Re\p_t\psi\\\Im\p_t\psi
\end{bmatrix}\in\R^4$.
The solitary wave
$\varPhi(x,\omega)e^{-\jj\omega t}\in\C^2$
is then represented by
$e^{\bfJ\omega t}\Phi(x,\omega)\in\R^4$,
where 
\[
\Phi(x,\omega)=
\begin{bmatrix}
\phi(x,\omega)\\0\\0\\-\omega\phi(x,\omega)
\end{bmatrix},
\qquad
\bfJ=
\begin{bmatrix}
        0     &      1      &  0    &    0\\ 
        -1     &       0      &  0    &    0\\
        0     &       0      &  0    &   1\\
        0     &       0      &  -1    &    0
\end{bmatrix}.
\]
We identify 
the perturbation
$\mathcal{X}\in\C^2$  with the real vector
$X=
\begin{bmatrix}
\Re\mathcal{X}_1
\\
\Im\mathcal{X}_1
\\
\Re\mathcal{X}_2
\\
\Im\mathcal{X}_2
\end{bmatrix}
\in\R^4$.
Then (\ref{lin3}) becomes
\begin{equation}\label{XX}
\dot X(x,t)=
\begin{bmatrix}0 &-\omega & 1& 0\\ \omega & 0 & 0 & 1\\
\p_x^2-m^2 & 0 & 0 &-\omega\\ 0&\p_x^2-m^2& \omega & 0\end{bmatrix}
 X(x,t)+\delta(x)\begin{bmatrix} 0\\0\\\alpha(1+2\kappa)X_{1}(0,t)\\
\alpha X_{2}(0,t)\end{bmatrix},
\end{equation}
where
(cf.~\eqref{a-kappa})
\begin{equation}\label{def-a-beta}
\alpha:=a(C^2)=2\varkappa>0,
%\qquad
%\beta:=C^2 a'(C^2),
\qquad
\kappa:=\fra{C^2 a'(C^2)}{a(C^2)}.
\end{equation}
This gives a system which is $\C$-linear.

\begin{remark}
We note that
the above definition of $\kappa$
is compatible with the pure power case
$a(\tau)=\tau^\kappa$,
$\kappa>0$,
$\tau\ge 0$,
when
$a'(\abs{\psi}^2)\abs{\psi}^2=\kappa a(\abs{\psi}^2)$.
\end{remark}

Let us denote
\begin{equation}\label{D}
\begin{array}{l}
%L_0(\omega)=-\p_x^2+m^2-\alpha\delta(x)-\omega^2,
%\\[1ex]
L_\kappa(\omega)=-\p_x^2+m^2-\alpha(1+2\kappa)\delta(x)-\omega^2.
\end{array}
\end{equation}

\begin{remark}
The domains of all operators which we consider
require a careful definition.
For example, $A=-\p_x^2+c\delta(x)$, $c\in\R$,
is defined as $A:\,L^2(\R)\to L^2(\R)$
with the domain
\[
\dom(A)=\{\psi\in H^2(\R_{-},\C)\cup H^2(\R_{+},\C),
\qquad
\p_x\psi(0+)-\p_x\psi(0-)=c\psi(0)
\},
\]
where it is selfadjoint.
\end{remark}

\begin{lemma}[Spectrum of $L_\kappa$]
\label{lemma-sigma-lpm}
Let $\omega\in(-m,m)$, $\kappa\in\R$.
The operator
$L_\kappa(\omega)$
is selfadjoint and satisfies
\begin{eqnarray*}
&&
\sigma\sb{\mathrm{ess}}(L_\kappa(\omega))=[\varkappa^2,+\infty),
\\[1ex]
&&
\sigma\sb{\mathrm{p}}(L_\kappa(\omega))=
\begin{cases}
\emptyset,&\kappa\le -1/2,
\\
\Lambda_\kappa(\omega):=-4\varkappa^2(\kappa+\kappa^2),
&\kappa>-1/2,
\end{cases}
\end{eqnarray*}
where $\varkappa=\sqrt{m^2-\omega^2}$.
The eigenvalue $\Lambda_\kappa(\omega)$ is simple.
\end{lemma}

\begin{proof}
Solving the equation
\[
(-\p_x^2-\alpha(1+2\kappa)\delta(x)+\varkappa^2)\psi=\Lambda\psi,
\qquad
\psi\in L^2(\R),
\]
we find that
$\Lambda$ has to satisfy $\Lambda<\varkappa^2$
and that
\begin{eqnarray}\label{l-psi}
\psi(x)=C e^{-\abs{x}\sqrt{\varkappa^2-\Lambda}},\qquad x\in\R,\quad C\ne 0,
\end{eqnarray}
while $\Lambda$
is obtained from the jump condition at $x=0$:
\[
\alpha(1+2\kappa)=2\sqrt{\varkappa^2-\Lambda}.
\]
This shows that there could only be an eigenvalue
if $\kappa>-1/2$.
Taking into account that $\alpha=2\varkappa$,
one arrives at
$\varkappa^2(1+2\kappa)^2=\varkappa^2-\Lambda$,
so
$\Lambda=-4\varkappa^2(\kappa+\kappa^2)$.
In particular,
setting $\kappa=0$, one obtains the point spectrum $\{0\}$ of $L_0$;
by \eqref{NEP},
the corresponding eigenvector is $\phi$.
\end{proof}

In terms of operators \eqref{D},
the system (\ref{XX}) reads as
\begin{equation}\label{lin4}
\dot X(x,t)=\bfA(\omega,\kappa) X(x,t),
\qquad
\bfA(\omega,\kappa):=
\begin{bmatrix}
0&-\omega&1&0\\
\omega&0&0&1\\
-L_\kappa(\omega)-\omega^2&0&0&-\omega\\
0&-L_0(\omega)-\omega^2&\omega&0
\end{bmatrix}.
\end{equation}
Theorem~\ref{locex} generalizes to equation (\ref{lin4}):
for every initial function $X(x,0)=X_0\in\calE$,
the equation admits a unique solution $X(x,t)\in C_b(\R,\calE)$.
Denote
\[
\bmSigma=\begin{bmatrix}0&I_2\\-I_2&0\end{bmatrix},
\qquad
\bfH(\omega,\kappa)=
\begin{bmatrix}
L_\kappa+\omega^2&0&0&\omega
\\
0&L_0+\omega^2&-\omega&0
\\
0&-\omega&1&0
\\
\omega&0&0&1
\end{bmatrix};
\]
then the operator $\bfA(\omega,\kappa)$ from \eqref{lin4}
corresponding to a linearization at the solitary wave
is factored into
\[
\bfA(\omega,\kappa)=\bmSigma\bfH(\omega,\kappa).
\]
Denote
\begin{eqnarray}
\bfG_1=\bfG_1^{-1}=\begin{bmatrix}
1&0&0&0\\
0&0&0&1\\
0&0&1&0\\
0&1&0&0
\end{bmatrix},
\qquad
\bfG_2=\bfG_2^{-1}=
\begin{bmatrix}
I_2&0\\0&\sigma_2
\end{bmatrix}
\end{eqnarray}
and consider the conjugated version of $\bfH(\omega,\kappa)$:
\begin{eqnarray}\label{def-tilde-h}
\tilde\bfH(\omega,\kappa)
=
\bfG_2\bfG_1\bfH(\omega,\kappa)\bfG_1^{-1}\bfG_2^{-1}
=
\begin{bmatrix}
L_\kappa+\omega^2&\omega&0&0\\
\omega&1&0&0\\
0&0&L_0+\omega^2&\omega\\
0&0&\omega&1\\
\end{bmatrix}
=:
\begin{bmatrix}
H_\kappa(\omega)&0\\
0&H_0(\omega)
\end{bmatrix}.
\end{eqnarray}

We start with the spectrum of
\begin{eqnarray}\label{def-h}
H_\kappa(\omega)
=
\begin{bmatrix}
L_\kappa+\omega^2&\omega
\\
\omega&1
\end{bmatrix}.
\end{eqnarray}

\begin{lemma}
\label{lemma-sigma-h-12}
For any
$\kappa\in\R$,
the essential spectrum of $H_\kappa(\omega)$
is given by
\begin{enumerate}
\item
\[
\sigma\sb{\mathrm{ess}}(H_\kappa(\omega))
=
\begin{cases}
[m^2,+\infty),&\omega=0,
\\
[c^{-}(\omega),1]\cup[c^{+}(\omega),+\infty),&\omega\in(-m,m)\setminus\{0\},
\end{cases}
\]
%% $\sigma\sb{\mathrm{ess}}(H_\kappa(\omega))
%% {\color{red}=[c^{-}(\omega),1]\cup [c^{+}(\omega),+\infty)$
%% in the case $\omega\ne 0$},
where 
\begin{eqnarray}\nonumber
%%&&
c^{\pm}(\omega)=\frac{m^2+1\pm\sqrt{(m^2-1)^2+4\omega^2}}2.
%%<\min\{1,m^2\},
%% \\
%% \nonumber
%% &&c^{+}(\omega)=\frac{m^2+1+\sqrt{(m^2-1)^2+4\omega^2}}2.
%%>\max\{1,m^2\}.
\end{eqnarray}
%% {\color{red}and  $\sigma\sb{\mathrm{ess}}(H_\kappa(0))=[m^2,\infty)$.}
We note that
$0<c^{-}(\omega)\le\min\{1,m^2\}$,
$c^{+}(\omega)\ge\max\{1,m^2\}$.
\item
\[
\sigma\sb{\mathrm{p}}(H_\kappa(\omega))
=
\begin{cases}
\emptyset,&\kappa\le -1/2,\quad\omega\in(-m,m)\setminus\{0\},
\\
\{1\},&\kappa\le -1/2,\quad \omega=0,
\\
\{\lambda^\pm_\kappa(\omega)\},&\kappa>-1/2,\quad\omega\in(-m,m),
\end{cases}
\]
with the eigenvalues $\lambda^\pm_\kappa(\omega)$ given by
\begin{eqnarray}\label{lambda-pm}
\lambda^\pm_\kappa(\omega)
=\fra{
\Big(
\Lambda_\kappa+\omega^2+1
\pm\sqrt{(\Lambda_\kappa+\omega^2+1)^2-4\Lambda_\kappa}\Big)}{2},
\end{eqnarray}
with $\Lambda_\kappa(\omega)=-4\varkappa^2(\kappa+\kappa^2)$
(when $\kappa>-1/2$)
from Lemma~\ref{lemma-sigma-lpm}.
In particular,
$0\in\sigma\sb{\mathrm{p}}(H_\kappa(\omega))$
if and only if $\kappa=0$,
and
\[
\sigma\sb{\mathrm{p}}(H_0(\omega))=\{0\}\cup\{\omega^2+1\},
\]
with the corresponding eigenvectors
\[
\psi_{0}(x)=
\begin{bmatrix}
\phi(x)\\-\omega\phi(x)
\end{bmatrix},
\qquad
\psi_{\omega^2+1}(x)=
\begin{bmatrix}
\omega\phi(x)\\\phi(x)
\end{bmatrix}.
\]
\end{enumerate}
\end{lemma}
%%%%%%%%%%%%%%%

%% \begin{remark}
%% Let us point out that
%% the expression under the square root in 
%% \eqref{lambda-pm}
%% is nonnegative.
%% This is immediate in the case $\Lambda_\kappa\le 0$;
%% if instead $\Lambda_\kappa>0$,
%% then
%% \[
%% (\Lambda_\kappa+\omega^2+1)^2-4\Lambda_\kappa
%% \ge
%% (\Lambda_\kappa+1)^2-4\Lambda_\kappa
%% =
%% (\Lambda_\kappa-1)^2\ge 0.
%% \]
%% {\color{red} EK:  etot Remark  ne nado :)}
%% \end{remark}

\begin{proof}
%\ac{ALL NEW PARAGRAPH:}
Let us compute the essential spectrum of
$H_\kappa(\omega)$.
Since
$\delta(x)$ is a relatively compact perturbation
of the Laplacian in one dimension,
by the Weyl theorem,
when computing the essential spectrum
of $H_\kappa(\omega)$,
we can replace $L_\kappa$
in \eqref{def-h}
by
$L_\kappa^{(0)}=-\p_x^2+m^2-\omega^2$.
Then the values of the essential spectrum
are those $\lambda\in\C$ such that
\begin{eqnarray}\label{sdf}
\det \begin{bmatrix}
\xi^2+\omega^2-\lambda &\omega
\\
\omega&1-\lambda
\end{bmatrix}
=(1-\lambda)\Big(\xi^2+m^2-\lambda-\frac{\omega^2}{1-\lambda}\Big)=0
\qquad
\mbox{for some $\xi\in\R$.}
\end{eqnarray}
Thus, either $\lambda=1$, or,
expressing $\xi^2$ in terms of $\lambda$, we arrive at the inequality
\begin{eqnarray}\label{asd0}
-m^2+\lambda+\frac{\omega^2}{1-\lambda}
=\frac{(\lambda-m^2)(1-\lambda)+\omega^2}{1-\lambda}
=\frac{\lambda^2-\lambda(m^2+1)+m^2-\omega^2}{\lambda-1}\ge 0.
\end{eqnarray}
There are two roots
$\lambda=c^{\pm}(\omega)$ of the numerator and
root $\lambda=1$ of the denominator;
these roots satisfy the inequalities
\[
0<c^{-}(\omega)<1<c^{+}(\omega),
\qquad
\omega\in(-m,m)\setminus\{0\}.
\]
(We note that when $\omega=0$,
one has $c^{-}(0)=1$ if $m^2\ge 1$
and
$c^{+}(0)=1$ if $m^2\le 1$.)
We conclude that the essential spectrum
%(the set where the inequality \eqref{sdf} is satisfied)
(the set where \eqref{sdf} is satisfied)
is given by
the intervals
$$
\lambda\in [c^{-}(\omega),1]\cup [c^{+}(\omega),+\infty).
$$

Let us now study the point spectrum of $H_\kappa(\omega)$.
First let us consider when $\lambda=1$ is an eigenvalue of
$H_\kappa(\omega)$.
One can see from \eqref{def-h}
that this is only possible 
when $\omega=0$
(with eigenfunctions
of the form $\begin{bmatrix}0\\u_2(x)\end{bmatrix}$,
$u_2\in L^2(\R)$);
in this case, the rest of the point spectrum
comes from Lemma~\ref{lemma-sigma-lpm}
(($\Lambda_\kappa(0)=-4m^2(\kappa+\kappa^2)$
in the case $\kappa>-1/2$),
agreeing with
$\lambda^{-}_\kappa(0)=1$
and
$\lambda^{+}_\kappa(0)=\Lambda_\kappa(0)$
from \eqref{lambda-pm}.

Now we consider the case
$\lambda\ne 1$. The operator
\begin{eqnarray}\label{def-h-1}
H_\kappa(\omega)-\lambda I_2
=
\begin{bmatrix}
L_\kappa+\omega^2-\lambda&\omega
\\
\omega&1-\lambda
\end{bmatrix}
\end{eqnarray}
has zero eigenvalue if and only if so does
the Schur complement of $1-\lambda$,
which is given by
\[
L_\kappa(\omega)+\omega^2
-\lambda-\frac{\omega^2}{1-\lambda}
=L_\kappa(\omega)
-\lambda-\frac{\lambda\omega^2}{1-\lambda}.
\]
Thus,
$\lambda\ne 1$ is an eigenvalue of $H_\kappa(\omega)$
if and only if
\begin{eqnarray}\label{lambda-lambda}
\Lambda=\lambda+\frac{\lambda\omega^2}{1-\lambda}
\end{eqnarray}
is an eigenvalue of
$L_\kappa(\omega)$;
see Figure~\ref{kg-delta-fig-spectrum-h}.
Now the proof follows from Lemma~\ref{lemma-sigma-lpm},
leading to the expressions \eqref{lambda-pm}.
Let us point out that the eigenfunctions
corresponding to $\lambda\ne 1$
are given by (cf. \eqref{l-psi})
\begin{eqnarray}\label{h-psi}
u\sp{\pm}(x)
=
\begin{bmatrix}
1-\lambda^\pm_\kappa
\\
-\omega
\end{bmatrix}
e^{-\abs{x}\sqrt{\varkappa^2-\Lambda}}.
\end{eqnarray}
\end{proof}

\begin{figure}[ht]
\begin{center}
\setlength{\unitlength}{1.3pt}
\begin{picture}(0,100)(25,-50)
\font\gnuplot=cmr10 at 10pt
\gnuplot
\put(0,-30){\circle*{3}}
\put(1,-38){$\Lambda_\kappa$}
\put(-27,4){$\lambda^{\!-}_{\!\kappa}$}
\put(-23,0){\circle*{3}}
\put(0,0){\circle*{3}}
\put(1,-9){$0$}
%\put(38,0){\circle*{3}}
\put(43,0){\circle*{3}}
\put(48.5,0){\circle*{3}}
\put(35,-7.5){$1$}
\put(40,4){$\lambda^{\!+}_{\!\kappa}$}
\put(48,-9){$\lambda^{\!+}_{\!0}$}
\put(-29,52){$\sigma(L_\kappa)$}
\put(90,5){$\sigma(H_\kappa(\omega))$}
\put(-13,29){$\varkappa^2$}
\put(18.5,3){$c^{-}$}
\put(58,3){$c^{+}$}
\put(64,35){$\Lambda=\lambda+\fra{\lambda\omega^2}{(1-\lambda)}$}
\put(-50, 0){\line(1,0){150}}
\put(0,-60,0){\line(0,1){120}}
\put(0,-60,0){\line(0,1){120}}

\multiput(-21,-30)(4.1,0){16}{\line(1,0){1}}
\multiput(-23,-30)(0,4){8}{\line(0,1){1}}
\multiput(43,-30)(0,4){8}{\line(0,1){1}}
\multiput(0,30)(3.99,0){15}{\line(1,0){1}}
\multiput(19.5,30)(0,-4){8}{\line(0,-1){1}}
\multiput(58.5,30)(0,-4){8}{\line(0,-1){1}}

\multiput(38,-60)(0,7){18}{\line(0,1){0.5}}
\multiput(-50,-60)(5,5){24}{\line(0,1){0.5}}

\linethickness{1pt}
\qbezier(-50,-59)(20,14)(32,60)
\qbezier(41,-60)(41,-0)(70,55)
\linethickness{3pt}
%\put(19.5,0){\line(1,0){92}}
\put(0,30){\line(0,1){30}}
%\linethickness{4pt}
\put(58.5,0){\line(1,0){59}}
\put(19.5,0){\line(1,0){18.5}}
\end{picture}
\end{center}
\caption{
\small
The relation between
$\sigma(L_\kappa)$
and
$\sigma(H_\kappa(\omega))$
for $\omega\ne 0$, $\kappa>0$.}
\label{kg-delta-fig-spectrum-h}
\end{figure}
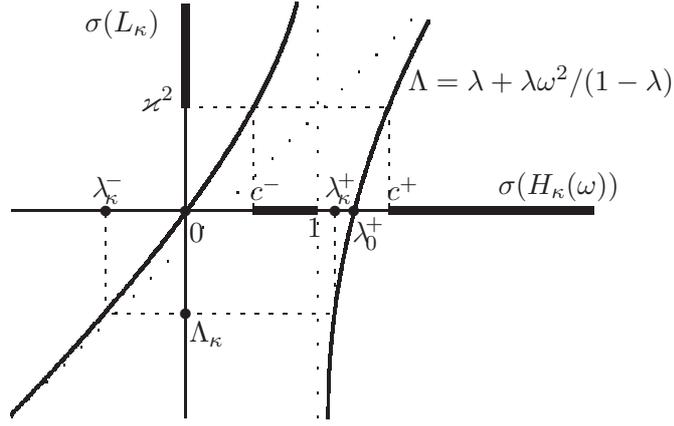

Lemma \ref{lemma-sigma-h-12}
gives the spectrum of the selfadjoint operator
$\bfH(\omega,\kappa):\,L^2(\R,\C^4)\to L^2(\R,\C^4)$:

\begin{corollary}
%\ac{MANY CORRECTIONS!!}
For any $\kappa\in\R$,
the essential spectrum of $\bfH(\omega,\kappa)$is given by
\[
\sigma\sb{\mathrm{ess}}(\bfH(\omega,\kappa))
=[c^{-}(\omega),1]\cup [c^{+}(\omega),\infty).
\]
%% \ac{WAS:
%% $\sigma\sb{\mathrm{ess}}(\bfH(\omega,\kappa))
%% =[c^{-}(\omega),1]\cup [m^2,\infty).
%% $
%% }
For $\kappa<0$, one has $\bfH(\omega,\kappa)\ge 0$,
with eigenvalue $\lambda=0$ of multiplicity one.
For $\kappa=0$, one has
$\bfH(\omega,0)\ge 0$,
with eigenvalue $\lambda=0$ of multiplicity two.
For $\kappa>0$ (when $\Lambda_\kappa(\omega)$
from Lemma~\ref{lemma-sigma-lpm}
is negative),
the operator $\bfH(\omega,\kappa)$ 
has exactly one simple negative eigenvalue  $\lambda_\kappa^{-}$
given by \eqref{lambda-pm}.
\end{corollary}
%%%%%%%%%%%%%%%%

The following theorem
gives the orbital stability result
for solitary wave solutions to
\eqref{KG}.

\begin{theorem}
\label{theorem-stability}
Assume that there are solitary wave solutions
for $\omega\in (\omega_1,\omega_2)\subset(-m,m)$.
The solitary wave
$\phi(x,\omega)e^{-\jj\omega t}$
is orbitally stable
if and only if
\[
\kappa<\frac{\omega^2}{m^2}.
\]
Above,
$\kappa:=\fra{\abs{\phi(0,\omega)}^2 a'(\abs{\phi(0,\omega)}^2)}
{a(\abs{\phi(0,\omega)}^2)}$
could be negative (see \eqref{def-a-beta}).
\end{theorem}

Let us recall that
for nonzero solitary waves $a(\abs{\phi(0,\omega)}^2)\ne 0$.

\begin{proof}
The Grillakis--Shatah--Strauss theory \cite{MR901236}
applies when
either
$\bfH\ge 0$, with $\lambda=0$ a simple eigenvalue
with the corresponding eigenvector $\bfJ\Phi$,
satisfying the assumptions of \cite[Theorem 1]{MR901236}.
(when $\kappa<0$),
or
$\bfH$ has a negative spectrum consisting of one simple
eigenvalue,
has its kernel spanned by $\bfJ\Phi$,
while the rest of its spectrum is positive and
separated away from $z=0$,
satisfying the assumptions of \cite[Theorem 2]{MR901236}.

In the case  $\kappa=0$, the operator  
$A_0(\omega)=\begin{bmatrix}
\jj\omega &1\\
-L_0(\omega)-\omega^2 &\jj\omega
\end{bmatrix}$
in the right-hand side of (\ref{lin3}) is  $\C$-linear.
Moreover, it can be represented as
$A_0(\omega)=J \tilde H_0(\omega)$, where
\[
J=\begin{bmatrix}
0 &1\\
-1& 0
\end{bmatrix},\quad  \tilde H_0(\omega)= \begin{bmatrix}
L_0(\omega)+\omega^2 &-\jj\omega\\
\jj\omega &1
\end{bmatrix}.
\]
The operator $\tilde H_0(\omega)$ has  the same spectral properties as the operator $H_0(\omega)$.
Namely, 
%% \ac{WAS:
%% $
%% \sigma\sb{\mathrm{ess}}\big(\tilde H_0(\omega)\big)
%% =[m,+\infty)
%% $
%% }
\[
\sigma\sb{\mathrm{ess}}\big(\tilde H_0(\omega)\big)
=[c^{-}(\omega),1]\cup[c^{+}(\omega),+\infty)
,\quad \sigma\sb{\mathrm{p}}\big(\tilde H_0(\omega)\big)
=\{0\}\cup\{\omega^2+1\},
\]
with eigenvalue $\lambda=0$ being simple.
Hence,
\cite[Theorems 1]{MR901236} applies, showing that
the solitary wave $\phi(x,\omega)e^{-\jj\omega t}$
is orbitally stable.

The instability in the case $\kappa>\omega^2/m^2$
follows from \cite[Theorem 3]{MR901236};
in the critical case, the instability follows from the
Jordan block structure of the zero eigenvalue;
see e.g. \cite{MR1995870}.
\end{proof}

The spectrum of the linearization at a solitary wave
also follows from the
arguments in \cite{kolokolov-1973} and \cite{MR901236};
we consider it next.

%%%%%%%%%%
\begin{lemma}
[Spectrum of $\bfA(\omega,\kappa)$: basic properties]
\label{lemma-sigma-a}
\quad
\begin{enumerate}
\item
\label{lemma-sigma-a-1}
If $\lambda$ belongs to $\sigma\sb{\mathrm{p}}(\bfA(\omega,\kappa))$,
then so do $\bar\lambda$, $-\lambda$, and $-\bar\lambda$.
\item
\label{lemma-sigma-a-2}
$\sigma\sb{\mathrm{ess}}(\bfA(\omega,\kappa))
=\jj\big(
\R\setminus(-m+\abs{\omega},m-\abs{\omega})
\big)$.
%%\ac{??????}
\item
\label{lemma-sigma-a-3}
$\sigma\sb{\mathrm{p}}(\bfA(\omega,\kappa))
\subset\R\cup\jj\R$.
There is a pair of a positive and a negative eigenvalues
of $\bfA(\omega,\kappa)$
if and only if
$\kappa>0$ and
$|\omega|<\min\{m,\Omega_\kappa\}$,
with
\begin{eqnarray}\label{def-o-k}
\Omega_\kappa:=m\sqrt{\kappa}.
\end{eqnarray}
\item
\label{lemma-sigma-a-4}
If
$\omega\in(-m,m)$
and
$\kappa\not\in\big\{0,\,\omega^2/m^2\big\}$,
then the zero eigenvalue of $\bfA(\omega,\kappa)$ is
of geometric multiplicity $1$
and
of algebraic multiplicity $2$.
%% with the generalized null space spanned by the vectors
%% $\bfJ\Phi$ and $\p_\omega\Phi$
%% which satisfy
%% \[
%% \bfA(\omega,\kappa)\bfJ\Phi=0,
%% \qquad
%% \bfA(\omega,\kappa)\p_\omega\Phi=\bfJ\Phi.
%% \]
\item
\label{lemma-sigma-a-5}
If
$\omega\in(-m,m)\setminus\{0\}$ and $\kappa=\omega^2/m^2$,
then
the zero eigenvalue of $\bfA(\omega,\kappa)$ is of
geometric multiplicity $1$
and
of algebraic multiplicity $4$.
\item
\label{lemma-sigma-a-6}
If
$\omega\in(-m,m)\setminus\{0\}$ and $\kappa=0$,
then
the zero eigenvalue of $\bfA(\omega,\kappa)$ is of
geometric multiplicity $2$
and
of algebraic multiplicity  $2$.
\item
\label{lemma-sigma-a-7}
If $\omega=0$ and $\kappa=0$,
then eigenvalue $z=0$ of $\bfA(\omega,\kappa)$ is of
geometric multiplicity $2$
and
of algebraic multiplicity $4$.
\end{enumerate}

\end{lemma}

\begin{proof}
We follow
the arguments of Kolokolov~\cite{kolokolov-1973}
and
Grillakis--Shatah--Strauss~\cite{MR901236}.
For Part~\itref{lemma-sigma-a-1},
one notices that
$\bfA(\omega,\kappa)$ has real coefficients
(hence $\bar\lambda$ is also an eigenvalue)
and then that
$\bfA(\omega,\kappa)^*=-\bfH(\omega,\kappa)\bmSigma$,
which is conjugate to $-\bmSigma\bfH(\omega,\kappa)$
(hence $-\bar\lambda$ is also an eigenvalue).
Part~\itref{lemma-sigma-a-2} follows
by fixing $\lambda\in\jj\R$ with
$\abs{\lambda}\ge m-\abs{\omega}$
and considering the Weyl sequences
supported away from $x=0$.
%%\ac{???}

%% from Remark~\ref{remark-a-ess}
%% and
%from the Weyl theorem
%on the essential spectrum.
%\medskip

To prove Part~\itref{lemma-sigma-a-3},
it is convenient to write $\bfA(\omega,\kappa)$
from \eqref{lin4}
in the form similar to the nonlinear Schr\"odinger
theory \cite{kolokolov-1973}.
We consider the conjugated versions
of $\bmSigma$ and $\bfH(\omega,\kappa)$,
with
\[
\tilde\bfH(\omega,\kappa)=\bfG_2\bfG_1\bfH(\omega,\kappa)\bfG_1^{-1}\bfG_2^{-1}
\]
from \eqref{def-tilde-h}
and with
\begin{eqnarray}\label{def-tilde-sigma}
\tilde\bmSigma=
\bfG_2\bfG_1\bmSigma\bfG_1^{-1}\bfG_2^{-1}
=-\jj\begin{bmatrix}
0&\sigma_1\\
\sigma_1&0
\end{bmatrix}.
\end{eqnarray}
Now we are in the framework
of A. Kolokolov~\cite{kolokolov-1973}.
Consider  the eigenvalue problem for $\tilde\bmSigma\tilde\bfH(\omega,\kappa)$:
\begin{eqnarray}\label{def-a-tilde}
\tilde\bmSigma\tilde\bfH(\omega,\kappa)\psi
=
-\jj\begin{bmatrix}0&\sigma_1\\\sigma_1&0\end{bmatrix}
\begin{bmatrix}H_\kappa(\omega)&0\\0&H_0(\omega)\end{bmatrix}
\begin{bmatrix}u\\v\end{bmatrix}
=\lambda
\begin{bmatrix}u\\v\end{bmatrix}, \qquad
\psi=
\begin{bmatrix}u\\v\end{bmatrix}\in L^2(\R,\C^4).
\end{eqnarray}
First, we notice that if $\kappa\le 0$,
then $\tilde\bfH(\omega,\kappa)$
is nonnegative and selfadjoint,
hence one can extract the square root
(also nonnegative and selfadjoint);
therefore,
\[
\sigma\sb{\mathrm{d}}(\tilde\bmSigma\tilde\bfH)\setminus\{0\}
=
\sigma\sb{\mathrm{d}}\big(\tilde\bfH^{1/2}\tilde\bmSigma\tilde\bfH^{1/2}\big)\setminus\{0\}
\subset\jj\R
\]
since $\tilde\bfH^{1/2}\tilde\bmSigma\tilde\bfH^{1/2}$
is antiselfadjoint.

From now on, we assume that $\kappa>0$.
Eliminating  $v$, we get
\[
-\sigma_1H_0(\omega)\sigma_1H_\kappa(\omega) u=\lambda^2 u, \qquad
u\in L^2(\R,\C^2).
\]
For $\lambda\ne 0$,
one can see that $u$ is orthogonal to
\begin{eqnarray}\label{kerlen-l0-sigma}
\ker(\sigma_1H_0(\omega)\sigma_1)=
\ker
\left(
\begin{bmatrix}1&\omega\\\omega&L_0+\omega^2\end{bmatrix}
\right)
=
\begin{bmatrix}-\omega\phi\\\phi\end{bmatrix}
\end{eqnarray}
(cf. \eqref{NEP}),
hence we can write
\[
H_\kappa(\omega) u=-\lambda^2(\sigma_1H_0(\omega)\sigma_1)^{-1} u
=-\lambda^2\sigma_1H_0(\omega)^{-1}\sigma_1 u.
\]
Coupling this relation with $u$ and taking into account that
$\langle u,(\sigma_1H_0(\omega)\sigma_1)^{-1} u\rangle>0$, 
we see that $\lambda^2\in\R$.
To find whether $-\lambda^2$ can be negative,
one considers the minimization problem
\begin{eqnarray}\label{x3}
\mu=
\inf\left\{
\langle u,H_\kappa(\omega) u \rangle
:\ \langle u,u\rangle=1,
\quad
u\in \big(\ker(\sigma_1H_0(\omega)\sigma_1)\big)^\perp
=
\begin{bmatrix}-\omega\phi\\\phi\end{bmatrix}
\sp\perp
\right\},
\end{eqnarray}
which implies that $u$ satisfies
\begin{eqnarray}\label{x33}
H_\kappa(\omega) u=\mu u
+
\nu
\begin{bmatrix}-\omega\phi\\\phi\end{bmatrix},
\end{eqnarray}
with $\mu,\,\nu\in\R$ the Lagrange multipliers.
We note that if $\mu\le 0$,
then $\nu\ne 0$, or else
one would have $\mu=\lambda^{-}_\kappa$ (the only negative
eigenvalue of $H_\kappa(\omega)$), which is not possible
since $u^{-}(x)$ from \eqref{h-psi}
corresponding to 
eigenvalue $\lambda^{-}_\kappa$ is not orthogonal
to $\ker(\sigma_1H_0(\omega)\sigma_1)$:
one has
\[
(u^{-})^*\begin{bmatrix}-\omega\phi\\\phi\end{bmatrix}
=
\begin{bmatrix}1-\lambda^{-}_\kappa\\-\omega\end{bmatrix}^*
\begin{bmatrix}-\omega\\1\end{bmatrix}
e^{-\abs{x}\sqrt{\varkappa^2-\Lambda}}
e^{-\abs{x}\varkappa}
=
-(2-\lambda^{-}_\kappa)\omega
e^{-\abs{x}\sqrt{\varkappa^2-\Lambda}}
e^{-\abs{x}\varkappa},
\]
where $2-\lambda^{-}_\kappa$ is strictly positive.
So, $\nu\ne 0$ (or else $\mu$ would have to be positive);
then one concludes from \eqref{x3} that $\mu>\lambda^{-}_\kappa$.
We rewrite \eqref{x33} as
\[
 (H_\kappa(\omega)-\mu )u
=
\nu
\begin{bmatrix}-\omega\phi\\\phi\end{bmatrix},
\qquad
u
=
\nu
(H_\kappa(\omega)-\mu )^{-1}
\begin{bmatrix}-\omega\phi\\\phi\end{bmatrix}.
\]
The sign of $\mu$
could be found from the condition that
$u$ is orthogonal to $\ker(\sigma_1H_0(\omega)\sigma_1)$:
\[
\left\langle
\begin{bmatrix}-\omega\phi\\\phi\end{bmatrix},
(H_\kappa(\omega)-z )^{-1}
\begin{bmatrix}-\omega\phi\\\phi\end{bmatrix}
\right\rangle=0.
\]
We consider
\[
h(z)=
\left\langle
\begin{bmatrix}-\omega\phi\\\phi\end{bmatrix},
(H_\kappa(\omega)-z )^{-1}
\begin{bmatrix}-\omega\phi\\\phi\end{bmatrix}
\right\rangle,
\qquad
z\in(\lambda^{-}_\kappa,c^{-})\subset\rho(H_\kappa(\omega)).
\]
Since $h(z)$ is monotonically increasing on  $\rho(H_\kappa(\omega))$,
the sign of $\mu$ is opposite to the sign of $h(0)$, 
which is given by
\begin{eqnarray}\label{h-zero}
h(0)=
\left\langle
\begin{bmatrix}-\omega\phi\\\phi\end{bmatrix},
H_\kappa(\omega)^{-1}
\begin{bmatrix}-\omega\phi\\\phi\end{bmatrix}
\right\rangle
=
\left\langle
\begin{bmatrix}-\omega\phi\\\phi\end{bmatrix},
\begin{bmatrix}-\p_\omega\phi
\\\omega\p_\omega\phi+\phi\end{bmatrix}
\right\rangle
=\p_\omega
\left(\omega\norm{\phi}^2\right)
.
\end{eqnarray}
Above, we used the relation
\begin{eqnarray}\label{theta-phi}
H_\kappa(\omega)\begin{bmatrix}
-\p_\omega\phi\\\omega\p_\omega\phi+\phi
\end{bmatrix}
=
\begin{bmatrix}L_\kappa+\omega^2&\omega\\\omega&1\end{bmatrix}
\begin{bmatrix}
-\p_\omega\phi\\\omega\p_\omega\phi+\phi
\end{bmatrix}
=
\begin{bmatrix}-\omega\phi\\\phi\end{bmatrix}
\end{eqnarray}
which in turn follows from
taking the $\omega$-derivative of
\eqref{NEP}
(after we substitute $\phi(x,\omega)$),
which yields
\begin{eqnarray}\label{relation-1}
L_\kappa\p_\omega\phi=2\omega\phi.
\end{eqnarray}
We conclude from \eqref{h-zero}
that there is $\lambda^2>0$ 
(hence,
there is a pair of a positive and a negative eigenvalues)
if and only if
$h(0)=\p_\omega(\omega \norm{\phi}^2)>0$.

Let us prove
Parts~\itref{lemma-sigma-a-4}
and~\itref{lemma-sigma-a-5}.
Given $\omega\in(-m,m)$,
let us compute the
geometric multiplicity of $\lambda=0$.
We consider $\tilde \bfA=\tilde\bmSigma\tilde\bfH(\omega,\kappa)$,
with $\tilde\bmSigma$ from
\eqref{def-tilde-sigma}
and $\tilde\bfH(\omega,\kappa)$ from \eqref{def-tilde-h}.
Since $\kappa\ne 0$,
by Lemma~\ref{lemma-sigma-h-12}
the geometric multiplicity
of $\lambda=0$ equals $1$,
with
\[
\ker(\tilde\bfH(\omega,\kappa))
=\Span\left\{
\phi(x)\bm{e}_3-\omega\phi(x)\bm{e}_4
\right\},
\]
with $\{\bm{e}_i\}_{1\le\jj\le 4}$ the standard basis in $\C^4$.

Now let us compute the
algebraic multiplicity of $\lambda=0$.
By
\eqref{kerlen-l0-sigma}
and
\eqref{theta-phi}
we have
\[
\tilde\bmSigma\tilde\bfH(\omega,\kappa)
\begin{bmatrix}
0\\0\\\phi\\-\omega\phi
\end{bmatrix}
=-\jj\begin{bmatrix}0&\sigma_1\\\sigma_1&0\end{bmatrix}
\begin{bmatrix}H_\kappa(\omega)&0\\0&H_0(\omega)\end{bmatrix}
\begin{bmatrix}
0\\0\\\phi\\-\omega\phi
\end{bmatrix}
=0,
\]
\begin{eqnarray}\label{asdf}
\tilde\bmSigma\tilde\bfH(\omega,\kappa)
\begin{bmatrix}
-\p_\omega\phi\\
\omega\p_\omega\phi+\phi
\\0\\0
\end{bmatrix}
=
-\jj\begin{bmatrix}0&\sigma_1\\\sigma_1&0\end{bmatrix}
\begin{bmatrix}H_\kappa(\omega)&0\\0&H_0(\omega)\end{bmatrix}
\begin{bmatrix}
-\p_\omega\phi\\
\omega\p_\omega\phi+\phi
\\0\\0
\end{bmatrix}
=
-\jj
\begin{bmatrix}
0\\0\\\phi\\-\omega\phi
\end{bmatrix};
\end{eqnarray}
thus $\lambda=0$ is an eigenvalue
of $\tilde\bmSigma\tilde\bfH(\omega,\kappa)$
of multiplicity at least two.
To be able to extend this Jordan chain,
solving
\begin{eqnarray}\label{wanted-xi}
\tilde\bmSigma\tilde\bfH(\omega,\kappa)
\Xi
=
\begin{bmatrix}
-\p_\omega\phi\\
\omega\p_\omega\phi+\phi
\\0\\0
\end{bmatrix},
\qquad
\Xi\in\C^4,
\end{eqnarray}
we need to make sure that
the right-hand side is orthogonal to the kernel of
the adjoint of the operator in the left-hand side,
\[
\ker\big(\tilde\bfH(\omega,\kappa)\tilde\bmSigma\big)
=
\Span
\left\{
\begin{bmatrix}
-\omega\phi\\\phi\\0\\0
\end{bmatrix}
\right\};
\]
thus, the condition
to have a Jordan block of a larger size is
\begin{eqnarray}\label{add}
0=
\left\langle
\begin{bmatrix}
-\omega\phi\\\phi\\0\\0
\end{bmatrix}
,
\ \begin{bmatrix}
-\p_\omega\phi\\\omega\p_\omega\phi+\phi\\0\\0
\end{bmatrix}
\right\rangle
=\p_\omega
\left(
\omega\norm{\phi}^2\right).
\end{eqnarray}
Let us compute explicitly the right-hand side in \eqref{add}.
Since
$\norm{\phi(\omega)}^2=\frac{C(\omega)^2}{\varkappa}$
(see Lemma~\ref{lemma-sw}),
one computes:
\begin{eqnarray}\label{the-requirement-0}
\frac{d}{d\omega}(\omega\norm{\phi(\omega)}^2)
=
\frac{C^2}{\varkappa}
+
\omega
\Big(
\frac{2C}{\varkappa}\frac{d C}{d\omega}
+\frac{\omega C^2}{\varkappa^3}
\Big)
=
\frac{1}{\varkappa^3}
\Big(
m^2 C^2-\frac{\omega^2 a(C^2)}{a'(C^2)}
\Big)
=
\frac{C^2}{\varkappa^3}
\Big(
m^2-\frac{\omega^2}{\kappa}
\Big),
\end{eqnarray}
where $C=C(\omega)$; to express $d C/d\omega$, we used \eqref{o-h}.
Thus, if $\kappa\ne 0$,
$\p\big(\omega\norm{\phi_\omega}^2\big)=0$
if and only if $\kappa=\omega^2/m^2$
(we will consider the case $\kappa=0$ separately).

If
$\kappa\ne \omega^2/m^2$,
$\kappa\ne 0$,
since the condition \eqref{add} is not satisfied,
equation \eqref{wanted-xi} has no $L^2$-solutions
hence the algebraic multiplicity is exactly two.
If
$\kappa=\omega^2/m^2>0$,
then equation \eqref{wanted-xi}
has an $L^2$-solution;
the algebraic multiplicity jumps by at least one.
This means that
an eigenvalue arrives at $z=0$;
by the symmetry of $\sigma(\bfA(\omega,\kappa))$
with respect to $\R$ and $\jj\R$,
there is also the opposite sign eigenvalue arriving at
$z=0$, thus the algebraic multiplicity of $z=0$
jumps by two.
Standard considerations (see e.g. \cite{MR1995870})
show that
for $\omega\ne\pm\Omega_\kappa$
the algebraic multiplicity
of zero eigenvalue
can not be more than four.
Indeed, let us show that
the equation
\[
(\tilde\bmSigma\tilde\bfH)^4\Gamma
=\phi(x)\e_3-\omega\phi(x)\e_4,
\qquad
\Gamma\in L^2(\R,\C^4),
\]
has no solutions.
%One can readily see that
%one would have $\Gamma(x)=\gamma_3(x)\e_3+\gamma_4(x)\e_4$.
If the algebraic multiplicity of $\lambda=0$ is at least four
(hence $\kappa=\omega^2/m^2$)
so that there are
$\Xi,\,\Theta\in L^2(\R,\C^4)$
such that
\begin{eqnarray}\label{as}
(\tilde\bmSigma\tilde\bfH)\Xi
=
-\p_\omega\phi\e_1+(\omega\p_\omega\phi+\phi)\e_2,
\qquad
(\tilde\bmSigma\tilde\bfH)\Theta=\Xi,
\end{eqnarray}
where, as one can readily see from the block form of
$\tilde\bmSigma\tilde\bfH$
(see \eqref{def-a-tilde}),
one can decompose
$\Xi(x)=\xi_3(x)\e_3+\xi_4(x)\e_4$,
%%and $\Theta(x)=\theta_1(x)\e_1+\theta_2(x)\e_2$,
then, trying to solve
\[
(\tilde\bmSigma\tilde\bfH)^2
\Gamma
=
\Xi,
\qquad
\Gamma
%(x)=\gamma_3(x)\e_3+\gamma_4(x)\e_4
\in L^2(\R,\C^4),
%%(\omega\p_\omega\phi+\phi)\e_3+\p_\omega\phi\e_4,
\]
we need to make sure that
$\Xi$
%%=\xi_3(x)\e_3+\xi_4(x)\e_4$
is orthogonal to the kernel of
$((\tilde\bmSigma\tilde\bfH)^2)^*
=(\tilde\bfH\tilde\bmSigma)^2$
which contains in particular
$\tilde\bmSigma(
(
-\p_\omega\phi\e_1
+\omega\p_\omega\phi+\phi)\e_2
)
=\tilde\bmSigma(\tilde\bmSigma\tilde\bfH)\Xi
=-\tilde\bfH\Xi
%=\tilde\bmSigma(\tilde\bmSigma\tilde\bfH)^2\Theta
$.
We arrive at the following necessary condition:
\begin{eqnarray}\label{ns1}
\langle\Xi,
\tilde\bfH(\omega,\kappa)\Xi\rangle
=0.
\end{eqnarray}
Taking into account that
$\tilde\bfH(\omega,\kappa)
=
\begin{bmatrix}
H_\kappa(\omega)&0
\\
0&H_0(\omega)
\end{bmatrix}
$
is semi-positive-definite on
vectors of the form
$\xi_3(x)\e_3+\xi_4(x)\e_4$,
with the one-dimensional kernel
%%(containing $\phi(x)\e_3-\omega\phi(x)\e_4$)
(this follows from Lemma~\ref{lemma-sigma-h-12}),
while by \eqref{as} the function $\Xi$ is not in this kernel,
we conclude that the left-hand side is strictly positive,
hence the condition
\eqref{ns1} can not be satisfied.

For Part~\itref{lemma-sigma-a-6},
we consider the case
$\omega\ne 0$,
$\kappa=0$.
By Lemma~\ref{lemma-sigma-h-12},
the geometric multiplicity
of $\lambda=0$ equals $2$,
with
\[
\ker\big(\tilde\bfH(\omega,0)\big)
=\Span\left\{
\phi(x)\bm{e}_1-\omega\phi(x)\bm{e}_2,\ 
\phi(x)\bm{e}_3-\omega\phi(x)\bm{e}_4
\right\}.
\]
To have a Jordan block,
we would need to solve
\begin{eqnarray}\label{ns}
\tilde\bmSigma\tilde\bfH(\omega,0)
\phi(x)\bm{e}_4
=
-\jj\begin{bmatrix}
0&\sigma_1\\\sigma_1&0
\end{bmatrix}
\begin{bmatrix}
H_0(\omega)&0\\0&H_0(\omega)
\end{bmatrix}
\Xi
=
-\jj
\phi(x)\bm{e}_1.
\end{eqnarray}
Since $\omega\ne 0$,
the right-hand side of \eqref{ns} is not orthogonal to
\[
\ker\big(
(\tilde\bmSigma\tilde\bfH(\omega,0))^*
\big)
=
\ker\big(
\tilde\bfH(\omega,0)\tilde\bmSigma
\big)
\ni\begin{bmatrix}-\omega\phi(x)\\\phi(x)\\0\\0\end{bmatrix}
,
\]
hence \eqref{ns} has no solutions.

Finally, for Part~\itref{lemma-sigma-a-7},
we consider the case
$\omega=0$,
$\kappa=0$.
In this case, the geometric multiplicity
of $\lambda=0$ equals $2$,
with
\[
\ker\big(\tilde\bfH(0,0)\big)
=\Span\left\{
\phi(x)\bm{e}_1,\ \phi(x)\bm{e}_3
\right\}.
\]
The Jordan block
corresponding to $\phi(x)\bm{e}_1$
is of size at least two since
\[
\tilde\bmSigma\tilde\bfH(0,0)
\phi(x)\bm{e}_4
=
-\jj\begin{bmatrix}
0&\sigma_1\\\sigma_1&0
\end{bmatrix}
\begin{bmatrix}
H_0(0)&0\\0&H_0(0)
\end{bmatrix}
\phi(x)\bm{e}_4
=
-\jj
\phi(x)\bm{e}_1.
\]
Now we notice that $\phi(x)\bm{e}_4$
is not orthogonal to
$\ker\big(\tilde\bfH(0,0)\tilde\bmSigma\big)$
which contains
$\phi(x)\bm{e}_4$;
therefore,
the corresponding Jordan block is of size exactly two.
Similarly, there is a Jordan block of size exactly two
corresponding to
$\phi(x)\bm{e}_3$.

This completes the proof of Lemma~\ref{lemma-sigma-a}.
\end{proof}

\begin{remark}
By Lemma~\ref{lemma-sigma-a}~\itref{lemma-sigma-a-3},
there is a positive eigenvalue of $\bfA(\omega,\kappa)$
if and only if
$\omega^2<\min\{m^2,m^2\kappa\}$.
We see from \eqref{the-requirement-0}
that in this case
one has $\p_\omega(\omega\norm{\phi(\omega)}^2)>0$.
Since the charge of the Klein--Gordon field
is given by
$\calQ(\phi(\omega))=\omega\norm{\phi(\omega)}^2$
(see \eqref{Q}),
the above is in agreement
with the Kolokolov stability condition
\begin{eqnarray}\label{K-C}
\p_\omega \calQ(\phi(\omega))<0
\end{eqnarray}
derived in \cite{kolokolov-1973}
in the context of the nonlinear Schr\"odinger equation.
\end{remark}

%%%%%%%%%%%%%%%%%%%%%%%%%%%%%%%%%%%
\section{The spectrum of the linearization operator}
\label{section-4}
%%%%%%%%%%%%%%%%%%%%%%%%%%%%%%%%

Now we are going to perform a complete analysis
of the spectrum of the linearization operator.
For  convenience,
we consider the linearization operator $\bfA(\omega,\kappa)$
from \eqref{lin4}
using $\kappa>0$ as a parameter.
Recall that
\begin{eqnarray}\label{def-a}
\bfA(\omega,\kappa)
=
\begin{bmatrix}
0&-\omega&1&0
\\
\omega&0&0&1
\\
\p_x^2-m^2+(1+2\kappa)\alpha\delta(x)&0&0&-\omega
\\
0&\p_x^2-m^2+\alpha\delta(x)&\omega&0
\end{bmatrix},
\qquad
\alpha=2\sqrt{m^2-\omega^2}.
\end{eqnarray}

For $x\ne 0$,
substituting $\Psi(x)=\w e^{-\nu |x|}$
with some $\nu\in\C$, $\Re\nu\ge  0$,
into the equation $(\bfA(\omega,\kappa)-\lambda )\Psi=0$, we get
\begin{equation}
\begin{bmatrix}
   -\lambda    &-\omega      &  1     &  0\\
   \omega     &     -\lambda     &  0     &  1\\
  \nu^2-m^2  &       0       & -\lambda  &-\omega\\
     0      &     \nu^2-m^2  & \omega   & -\lambda
\end{bmatrix}
\w=0,
\qquad
\nu\in\C,\quad\w\in\C^4,\quad
\w\ne 0,
\end{equation}
which,
via the Schur complement idea,
is equivalent to
\begin{equation}\label{s-v}
\begin{bmatrix}
\begin{array}{cc}-\lambda&-\omega\\\omega&-\lambda\end{array}
&
\begin{array}{cc}1&0\\0&1\end{array}
\\
S(\omega,\lambda,\nu)
&
\begin{array}{cc}0&0\\0&0\end{array}
\end{bmatrix}
\w=0,
\qquad
\w\in\C^4,\quad
\w\ne 0,
\end{equation}
with
$S(\omega,\lambda,\nu)\in\mathop{\mathrm{End}}(\C^2)$
the Schur complement
of the top right block $I_2$:
\begin{eqnarray}\label{def-s}
S(\omega,\lambda,\nu)=
\begin{bmatrix}\nu^2-m^2&0\\0&\nu^2-m^2\end{bmatrix}
-
\begin{bmatrix}-\lambda&-\omega\\\omega&-\lambda\end{bmatrix}^2
=
\begin{bmatrix}\nu^2-m^2+\omega^2-\lambda^2&-2\lambda\omega
\\2\lambda\omega&\nu^2-m^2+\omega^2-\lambda^2\end{bmatrix}.
\end{eqnarray}
The condition to have nonzero solution $\w\in\C^4$
to \eqref{s-v}
is equivalent to
\[
\det S(\omega,\lambda,\nu)
=
(m^2-\omega^2+\lambda^2-\nu^2)^2+4\lambda^2\omega^2=0.
\]
This gives
\[
m^2-\omega^2+\lambda^2-\nu^2=-2\jj\lambda\omega,
\qquad
m^2-\omega^2+\lambda^2-\nu^2=2\jj\lambda\omega,
\]
\[
m^2-(\omega-\jj\lambda)^2=\nu^2,
\qquad
m^2-(\omega+\jj\lambda)^2=\nu^2,
\]
allowing one to express $\nu\in\C$
in terms of $\omega$ and $\lambda$:
\[
\nu=\sqrt{m^2-(\omega\pm \jj\lambda)^2}.
%\qquad
%\Re\nu_\pm> 0.
\]
We choose the cuts in the complex plane $\lambda$ from the branching points to
infinity:
\begin{eqnarray*}
\calC_+:=(-\jj\infty,-\jj (m-\omega)]\cup[\jj (m+\omega,\jj \infty),
\\
\calC_-:=(-\jj\infty,-\jj (m+\omega)]\cup[\jj (m-\omega,\jj \infty),
\end{eqnarray*}
defining
\begin{eqnarray}\label{def-nu}
\nu_\pm(\omega,\lambda)=\sqrt{m^2-(\omega\pm \jj\lambda)^2},
\end{eqnarray}
with
\begin{equation}\label{re}
\Re\nu_\pm(\omega,\lambda)>0,
\qquad
\lambda\in\C\setminus \calC_\pm.
\end{equation}

The zero eigenvalue of the Schur complement $S$
from \eqref{def-s}
corresponds to
two eigenvectors
$u_\pm\in\C^2$,
depending on the choice $\nu=\nu_\pm$;
these eigenvectors
are given by
\[
\begin{bmatrix}
2\lambda\omega
\\
-m^2+\omega^2+\nu_\pm^2-\lambda^2
\end{bmatrix}
=
\begin{bmatrix}
2\lambda\omega
\\
-m^2+\omega^2+(m^2-\omega^2+\lambda^2\mp 2\jj\lambda\omega)-\lambda^2
\end{bmatrix}
=
2\lambda\omega
\begin{bmatrix}1\\\mp \jj\end{bmatrix},
\]
so we can use
$
u_\pm
=\begin{bmatrix}1\\\mp \jj\end{bmatrix}$.
By \eqref{s-v},
the corresponding vector
from the null space of
$\bfA(\omega,\kappa)-\lambda I$
is then
$\w_\pm=\begin{bmatrix}u_\pm\\v_\pm\end{bmatrix}\in\C^4$,
with
$
v_\pm
=
-
\begin{bmatrix}
-\lambda&-\omega\\\omega&-\lambda\end{bmatrix}u_\pm
=
\begin{bmatrix}
\lambda\mp\jj\omega
\\
-\omega\mp\jj\lambda
\end{bmatrix}\in\C^2.
$
Thus, one has
\[
\w_\pm=
\begin{bmatrix}
1
\\
\pm\jj
\\
\lambda\mp\jj\omega
\\
\omega\pm\jj\lambda
\end{bmatrix}
\in\C^4.
\]
Therefore, an eigenfunction
corresponding to the eigenvalue $\lambda\in\C$
is of the form
\begin{eqnarray}\label{eigenfunction}
\Psi(x,\omega,\lambda)
=
A\w_{+}e^{-\nu_{+}\abs{x}}
+
B\w_{-}e^{-\nu_{-}\abs{x}}
=
A\begin{bmatrix}
1
\\
\jj
\\
\lambda-\jj\omega
\\
\omega+\jj\lambda
\end{bmatrix}
e^{-\nu_+\abs{x}}
+
B\begin{bmatrix}
1
\\
-\jj
\\
\lambda+\jj\omega
\\
\omega-\jj\lambda
\end{bmatrix}
e^{-\nu_{-}\abs{x}},
\end{eqnarray}
with $A,\,B\in\C$ not simultaneously zeros and $\Re\nu_{\pm}>0$
(when the corresponding coefficient is nonzero).
The values of $A$ and $B$ are obtained from the jump
conditions at $x=0$:
substituting $\Psi$ into
$(\bfA(\omega,\kappa)-\lambda )\Psi=0$
and collecting the terms with $\delta$-function
gives:
\begin{eqnarray}\label{a-vs-b}
\begin{cases}
(-2\nu_{+}+\alpha(1+2\kappa))A
+(-2\nu_{-}+\alpha(1+2\kappa))B=0,
\\
(-2\nu_{+}+\alpha)\jj A
+(-2\nu_{-}+\alpha)(-\jj B)=0.
\end{cases}
\end{eqnarray}
The condition to have
$A,\,B\in\C$ not simultaneously zeros,
\[
\det
\begin{bmatrix}
-2\nu_{+}+\alpha(1+2\kappa)
&
-2\nu_{-}+\alpha(1+2\kappa)
\\
-2\nu_{+}+\alpha
&
2\nu_{-}-\alpha
\end{bmatrix}
=0,
\]
takes the form
$
-8\nu_{+}\nu_{-}
+
4(\nu_{+}+\nu_{-})\alpha(1+\kappa)
-2\alpha^2(1+2\kappa)
=0,
$
which we rewrite as
\begin{eqnarray}\label{deter-nu}
D_{\omega,\kappa}(\lambda)=0
\end{eqnarray}
with
\begin{eqnarray}\label{def-d-lambda}
D_{\omega,\kappa}(\lambda)=
\alpha^2(1+\kappa)^2
-
2(\nu_{+}+\nu_{-})\alpha(1+\kappa)
+4\nu_{+}\nu_{-}
-\alpha^2\kappa^2,
\end{eqnarray}
with
$\alpha=2\sqrt{m^2-\omega^2}$
from \eqref{def-a-beta}
and $\nu\sb\pm(\omega,\lambda)=\sqrt{m^2-(\omega\pm\jj\lambda)^2}$
from \eqref{def-nu}.
The function $D_{\omega,\kappa}(\lambda)$ is analytic
in $\lambda$
in $\C\setminus (\calC_-\cup\calC_+)$.
Since there are two possible
values for each of the square roots in
the definition \eqref{def-nu} of $\nu_\pm$,
$D_{\omega,\kappa}(\lambda)$
can be continuied analytically
through the cuts $\calC_-$ and $\calC_+$
to an analytic function
on the four-sheet cover of $\C$,
which we also denote by $D_{\omega,\kappa}(\lambda)$.
We call
the sheet defined by conditions
(\ref{re}) the \emph{physical sheet} of $D_{\omega,\kappa}(\lambda)$.
%%%%%%%%%%%%%%%%%%%%%%%%%%%%%%%%%%%%%%%%%%%%%%%%%%%%%%%%%
\subsection{Embedded eigenvalues and  virtual levels}
%%%%%%%%%%%%%%%%%%%%%%%%%%%%%%%%%%%%%%%%%%%%%%%%%%%%%%%%%%
We start by studying embedded eigenvalues and virtual levels of the operator $\bfA(\omega,\kappa)$.
Before formulating our results,
let us mention that
a \emph{virtual level}
(also known as a \emph{threshold resonance})
can be defined as a limit point of an eigenvalue family
which corresponds to values of a perturbation parameter in an interval
when this limit point no longer corresponds to a square-integrable
eigenfunction.
The virtual levels usually occur
at thresholds of the essential spectrum
(the endpoints of the essential spectrum or
the points where the continuous spectrum
changes its multiplicity).
For more on the phenomenon of virtual levels,
see e.g. \cite{MR544248,MR1841744,MR2598115,gesztesy2020absence,erdogan2019dispersive}.

%%%%%%%%%%%%%%%%%%%%
\begin{lemma}[Embedded eigenvalues and virtual levels of
$\bfA(\omega,\kappa)$]

\label{lemma-vl}
\quad
\begin{enumerate}
\item
\label{lemma-vl-1}
There are  embedded eigenvalues $\lambda=\pm 2\omega\jj$ if and only if  $\kappa=0$ and $|\omega|\ge m/3$.
\item
\label{lemma-vl-2}
For any $\kappa\in\R$ and  $\omega\in (-m,m)\setminus 0$ there are no virtual levels
at the embedded thresholds $\lambda=\pm\jj(m+\abs{\omega})$.
\item
\label{lemma-vl-3}
There are virtual levels at $\lambda=\pm\jj(m-\abs{\omega})$
if and only if $\kappa\in\big[-\frac{1}{2},\frac{1}{\sqrt 2}\big)$
and $\omega=\pm\calT_\kappa$,
where
\begin{eqnarray}\label{def-omega-omega}
\calT_\kappa=m\frac{(1+2\kappa)^2}{3+4\kappa};
\end{eqnarray}
in particular,  $\lambda=\pm\jj m$ are virtual levels if and only if $\kappa=-1/2$ and $\omega=0$.
By Part~\ref{lemma-vl-1},
these are genuine virtual levels
(with non-$L^2$ eigenfunction)
if and only if $\kappa\ne 0$.
\end{enumerate}
\end{lemma}
%%%%%%%%%%%%%%%%%%%%%%%%%%%
\begin{proof}
Because of the symmetry with respect to the sign
of $\omega$,
we will  assume  that
\[
0\le \omega<m.
\]

For Part~\itref{lemma-vl-1},
it suffices to
consider  $\lambda=\jj\Lambda$ with $\Lambda\ge m-\omega$.
In this case
\[
\nu_{-}=\sqrt{m^2-(\omega+\Lambda)^2}\in\jj\R;
\]
therefore,
in the expression for
the corresponding eigenfunction \eqref{eigenfunction}
one would need to take $B=0$.
There are two cases to consider:
\begin{itemize}
\item
If $\kappa\ne 0$,
the system \eqref{a-vs-b}
shows that one would also have $A=0$;
hence, $\lambda$ can not be an eigenvalue;
\item
If $\kappa=0$ and $\omega\in[m/3,m)$,
there is an embedded eigenvalue $\lambda=2\omega\jj$
since the system
\eqref{a-vs-b}
is satisfied for $B=0$ and any $A\in\C$.
(Note that in this case $\Lambda=2\omega\ge m-\omega$
and
$\nu_+=\varkappa=\alpha/2$.)
\end{itemize}
%% Note that in the case $\kappa=0$ and $\omega=m/3$,
%% $\Lambda=2m/3$
%% \ac{???}

For Part~\itref{lemma-vl-2},
it is enough to consider
$\omega\in(0,m)$, $\lambda=\jj(m+\omega)$.
At this value of $\lambda$, one has
$\nu_{+}=0$,
 $\nu_{-}=\sqrt{m^2-(m+2\omega)^2}\in \jj\R\setminus \{0\}$.
Thus, 
$$
D_{\omega,\kappa}(\jj(m+\omega))
=\alpha^2(1+2\kappa)-2\nu_{-}\alpha(1+\kappa)
$$
can not be zero since  $\alpha>0$.
In the case $\omega=0$, one has $\nu_{\pm}=0$,
and then $D_{\omega,\kappa}(\jj m)=0$ if and only if $\kappa=-1/2$.

For Part~\itref{lemma-vl-3},
we need to consider
$\omega\in(0,m)$, $\lambda=\jj(m-\omega)$.
In this case
\[
\nu_{-}=0,
\quad
\nu_+=\sqrt{m^2-(2\omega-m)^2}
=2\sqrt{\omega(m-\omega)}>0
\]
by \eqref{re}. We need to solve:
\[
0=D_{\omega,\kappa}(\jj(m-\omega))
=
\alpha^2(1+2\kappa)-2\nu_{+}\alpha(1+\kappa).
\]
The equation  has no solution for $\kappa=-1$. In the case $\kappa\not=-1$, it is equivalent to
\[
\sqrt{m^2-\omega^2}\frac{1+2\kappa}{1+\kappa}=2\sqrt{\omega(m-\omega)},
\]
which leads to
\begin{eqnarray}\label{cT}
\omega=\calT_\kappa=m\frac{(1+2\kappa)^2}{3+4\kappa}=m\Big(\kappa+\frac 14+\frac{1}{4(3+4\kappa)}\Big).
\end{eqnarray}
We note that $0\le\calT_\kappa<m$ as long as
$\kappa\in[-1/2,1/\sqrt{2})$.
\end{proof}

\subsection{Discrete spectrum}

Now we will study  the discrete spectrum
of the operator $(\bfA(\omega,\kappa))$.
Recall that for any $\kappa\in\R$ and $\omega\in(-m,m)$ the operator  has zero eigenvalue of  of multiplicity $2$.
We start with the special  cases $\kappa=0$ and $\omega=0$.
%%%%%
\begin{lemma}[Eigenvalues of $\bfA(\omega,0)$, $\omega\ne 0$]
\label{K0}
$$
\sigma\sb{\mathrm{p}}(\bfA(\omega,0))=\{0,\pm 2\omega\jj\}\quad 0<|\omega|<m,
$$
with eigenvalues
$\lambda=\pm 2\omega\jj$ being simple.
The only eigenvalue of  $\bfA(0,0)$
is $\lambda=0$ of algebraic multiplicity $4$
and geometric multiplicity $2$.
\end{lemma}
%%%%%%%
\begin{proof}
Considering
for which values  $\lambda$
the matrix
\begin{eqnarray}\label{a0}
\bfA(\omega,0)-\lambda I_4=
\begin{bmatrix}
-\lambda&-\omega&1&0
\\
\omega&0-\lambda&0&1
\\
-L_0(\omega)-\omega^2&0&-\lambda&-\omega
\\
0&-L_0(\omega)-\omega^2&\omega&-\lambda
\end{bmatrix}
\end{eqnarray}
has eigenvalue zero
reduces to studying this question
for the Schur complement of the top right block $I_2$,
\begin{eqnarray}\label{def-ss}
S=
\begin{bmatrix}-L_0(\omega)-\omega^2&0\\0&-L_0(\omega)-\omega^2\end{bmatrix}
-
\begin{bmatrix}-\lambda&-\omega\\\omega&-\lambda\end{bmatrix}^2
=
\begin{bmatrix}-L_0(\omega)-\lambda^2&-2\lambda\omega\\2\lambda\omega&-L_0(\omega)-\lambda^2\end{bmatrix}.
\end{eqnarray}
By Lemma~\ref{lemma-sigma-lpm},
$L_0(\omega)$ has the only eigenvalue $0$;
we conclude that $\lambda$ can be obtained from the equation
\[
\det\begin{bmatrix}-\lambda^2&-2\lambda\omega
\\2\lambda\omega&-\lambda^2\end{bmatrix}=0,
\]
which gives $\lambda=0$ and $\lambda=\pm 2\omega\jj$.
Recall that the eigenvalues $\pm 2\omega\jj$
are embedded into the essential spectrum
if $\abs{\omega}\ge m/3$.
In the case $\omega=0$,
the geometric multiplicity could be obtained from substituting
$\lambda=0$, $\omega=0$, and $\kappa=0$ into \eqref{a0}
and taking into account that
$\Lambda_0(\omega)=0$ is a simple eigenvalue of $L_0(\omega)$
(Lemma~\ref{lemma-sigma-lpm}).
\end{proof}
%%%%%%%%%%%%%%%%%%%%
\begin{lemma} [Eigenvalues of $\bfA(0,\kappa)$]
\label{omega0}
$$
\sigma\sb{\mathrm{p}}(\bfA(0,\kappa))=
\left\{\begin{array}{ll}
0,\quad\quad \kappa\le -1/2,\\
\pm 2m\sqrt{\kappa(1+\kappa)},\quad \kappa>-1/2.
\end{array}\right.
$$
The  eigenvalues $\lambda=\pm 2m\sqrt{\kappa(1+\kappa)}$ are simple
(real if $\kappa>0$ and purely imaginary  if $-1/2<\kappa<0$).
\end{lemma}
%%%%%%%%%%%%%%%%%%%%%%%%%%%%
\begin{proof}
In the case $\omega=0$,
one has $\nu_{\pm}=\sqrt{m^2+\lambda^2}$ and $\alpha=2m$. Then
\begin{eqnarray}\nonumber
D_{0,\kappa}(\lambda)&=&4m^2(1+2\kappa)-8 m(1+\kappa)\sqrt{m^2+\lambda^2}+4(m^2+\lambda^2)\\
\nonumber
&=&8m^2(1+\kappa)+4\lambda^2-8 m(1+\kappa)\sqrt{m^2+\lambda^2}.
\end{eqnarray}
In the case $\kappa=-1$, $D_{0,\kappa}(\lambda)$ has the only root $\lambda=0$
of multiplicity $2$. For $\kappa\not=-1$,
$D_{0,\kappa}(\lambda)=0$ if and only if
\begin{equation}\label{om0-c}
2m^2+\frac{\lambda^2}{1+\kappa}>0 ~~{\rm and}~~\Big(2m^2+\frac{\lambda^2}{1+\kappa}\Big)^2=4m^2(m^2+\lambda^2). 
\end{equation}
The second condition gives $\lambda^2=0$ and $\frac{\lambda^2}{1+\kappa}=4m^2\kappa$.
Using the first condition of \eqref{om0-c},
we conclude that $\kappa>-1/2$, and then $\lambda=2m\sqrt{\kappa(1+\kappa)}$.
\end{proof}
Denote
\begin{eqnarray}\label{def-k-omega}
K_{\omega}=\frac{\sqrt{m^2-(2|\omega|-m)^2}-\varkappa}{2\varkappa-\sqrt{m^2-(2|\omega|-m)^2}}
=\frac{2\sqrt{\frac{|\omega|}{m+|\omega|}}-1}{2-2\sqrt{\frac{|\omega|}{m+|\omega|}}},
\qquad\omega\in(-m,m).
\end{eqnarray}
We note that
\begin{equation}\label{Kom}
K_{\omega}\le 0~~{\rm for}~~\omega\le m/3,\quad {\rm and}~~K_{\omega}\ge 0~~{\rm for}~~\omega\ge m/3.
\end{equation}
We note also that for $\omega\in[0,m)$ the function
$\omega\mapsto K_\omega$
is the inverse to the function
$\kappa\mapsto\calT_\kappa$
defined in \eqref{def-omega-omega}.
%%%%%%%%%%%%%%%%%%%%%%%%%%%%%%%%%%%%%%%%%%%
\begin{lemma}
\label{lemma-negative}
For $\omega\in(-m,m)$ and $\kappa\in (-\infty,0)\cup (0,K_{\omega})$,
$\sigma\sb{\mathrm{p}}(\bfA(\omega,\kappa))=\{0\}$.
 \end{lemma}
%%%%%%%%%%%%%%%%%%%%%%%%%%
\begin{proof}
For  $\omega=0$, the statement follows from Lemma \ref{omega0};
because of the symmetry with respect to the sign of $\omega$,
it suffices to consider
$\omega\in(0,m)$.
Due to  Lemma~\ref{lemma-sigma-a}~\itref{lemma-sigma-a-3},
it suffices to prove
that 
$$
\sigma\sb{\mathrm{p}}(\bfA(\omega,\kappa))\cap \jj\R =\{0\},
\qquad
\kappa\in (-\infty,0)\cup (0,K_{\omega}),\quad  \omega\in(0,m).
$$ 
It suffices to prove that
\begin{equation}\label{d}
D_{\omega,\kappa}(\lambda)\ne 0,\quad \lambda\in  \jj\R\setminus \{0\}, \qquad
\kappa\in (-\infty,0)\cup (0,K_{\omega}), \quad \omega\in(0,m),
\end{equation}
with $D_{\omega,\kappa}(\lambda)$ defined in \eqref{def-d-lambda}.
We rewrite  $D_{\omega,\kappa}(\lambda)$ as
\begin{equation}\label{d1}
D_{\omega,\kappa}(\lambda)= (\alpha-2\nu_{+})(\alpha-2\nu_{-})+2\kappa\alpha(\alpha-\nu_+-\nu_-).
\end{equation}
Note that
$\lambda=\pm 2\omega\jj$ are not roots of
$D_{\omega,\kappa}(\lambda)\ne 0$ for $\omega>0$ and $\kappa\ne 0$.
Indeed, for $\lambda=2\omega\jj$ one has $\alpha-2\nu_{+}=0$, and then
\[
D_{\omega,\kappa}(2\omega\jj)
=\alpha\kappa(\alpha-2\nu_-)
=\alpha\kappa
\Big(
2\sqrt{m^2-\omega^2}-2\sqrt{m^2-9\omega^2}
\Big)\ne 0;
\]
the case $\lambda=-2\omega\jj$ is treated similarly.

For $\lambda\ne\{0,\pm 2\omega\jj\}$,
%since $\alpha=2\sqrt{m^2-\omega^2}$
%is different from either of
%$2\nu\sb\pm=2\sqrt{m^2-(\omega\pm\jj\lambda)^2}$,
we may rewrite \eqref{d1} as
$$
D_{\omega,\kappa}(\lambda)= (\alpha-2\nu_{+})(\alpha-2\nu_{-})\Big(1+\kappa\alpha\Big( \frac{1} {\alpha-2\nu_{+}}+ \frac{1} {\alpha-2\nu_{-}}\Big)\Big).
$$
It remains to prove that for
$\omega\in(0,m)$ and
$\kappa\in (-\infty,0)\cup (0,K_{\omega})$, the equation
\[
1+\kappa\alpha\Big( \frac{1} {\alpha-2\nu_{+}}+ \frac{1} {\alpha-2\nu_{-}}\Big)=0
\]
has no solutions $\lambda\in  \jj\R\setminus \{0,\pm 2\omega\jj\}$.
Denoting $\Lambda=\jj\lambda$, we rewrite  the above equation  as
\begin{equation}\label{eq1}
1+\kappa\varkappa Q(\Lambda)=0,
\end{equation}
where
\begin{eqnarray}\label{r-lambda}
Q(\Lambda)
&=&\frac{1} {\sqrt {m^2-\omega^2}-\sqrt{m^2-(\omega+\Lambda)^2}}+ \frac{1} {\sqrt {m^2-\omega^2}-\sqrt{m^2-(\omega-\Lambda)^2}}
\nonumber
\\
&=&
\frac{1}{\Lambda}\biggr[\frac{\varkappa+\sqrt{\varkappa^2-\lambda^2-2\omega\Lambda}}{\Lambda+2\omega}
+ \frac{\varkappa+\sqrt{\varkappa^2-\Lambda^2+2\omega\Lambda}}{\Lambda-2\omega}\biggr].
\end{eqnarray}
%% Because of the symmetry with respect to the sign
%% of $\omega$,
%% we will assume that
%% $0<\omega<m$.
Due to (\ref{Kom}) it suffices to prove that there are no solution to  (\ref{eq1}) in the following domains:

{\it (1)} $0<\omega<m/3$, $\ \kappa<T_\kappa(\omega)<0$, 
$\ 2\omega<\Lambda< m-\omega$;

{\it (2)}
$0<\omega<m/3$, $\ \kappa<T_\kappa(\omega)<0$,  $\ 0<\Lambda< 2\omega$;

{\it (3)}
$m/3<\omega<m$, $\ 0<\kappa< T_\kappa(\omega)$, $\ 0<\Lambda< m-\omega$.

\smallskip
\noindent
{\it (1)}\quad
Note that $Q(\Lambda)>0$ for $2\omega<\Lambda< m-\omega$, and  
$$
Q'(\Lambda)=-\sum\limits_{\pm}\frac{\Lambda\pm \omega}{\sqrt{m^2-(\Lambda\pm \omega)^2}\,\Big(\varkappa-\sqrt{m^2-(\Lambda\pm \omega)^2}\Big)^2}<0. 
$$
Hence,
\begin{eqnarray}\nonumber
\inf\limits_{(2\omega,m-\omega)} Q(\Lambda)&=&Q(m-\omega)=\frac{1} {\varkappa}+ \frac{1} {\varkappa-\sqrt{m^2-(2\omega-m)^2}}
=\frac{1}{\varkappa}\frac{2\sqrt{m^2-\omega^2}-2\sqrt{\omega(m-\omega)}} {\sqrt{m^2-\omega^2}-2\sqrt{\omega(m-\omega)}}\\
\label{Qmin}
&=&\frac{1}{\varkappa}\frac{2\sqrt{m+\omega}-2\sqrt{\omega}} {\sqrt{m+\omega}-2\sqrt{\omega}}
=\frac{1}{\varkappa}\frac{2-2\sqrt{\frac{\omega}{m+\omega}}} {1-2\sqrt{\frac{\omega}{m+\omega}}}
=\frac{1}{\varkappa}|K_\omega|^{-1}>0.
\end{eqnarray}
and  
\[
|\kappa|\varkappa Q(\Lambda)>
|K_\omega|\varkappa Q(m-\omega)=1.
\]
Therefore (\ref{eq1}) has no solutions in the first domain.

\smallskip

\noindent
{\it (2)}\quad
Note that  $0<\Lambda<\min\{ 2\omega,m-\omega\}$ in the second and  the third domains. Hence 
$|\Lambda(\Lambda\pm 2\omega)|<\varkappa^2$, and  the  following expansion holds:
\beqn\nonumber
\sqrt{m^2-(\omega\pm\Lambda)^2}&=&\sqrt{\varkappa^2-\Lambda(\Lambda\pm 2\omega)}
=\varkappa
\bigg(
1-
\frac{\Lambda(\Lambda\pm 2\omega)}{2\varkappa^2}
-
\sum_{n=2}^\infty
\frac{(2n-3)!!
\Lambda^n(\Lambda\pm 2\omega)^n}{n!2^n\varkappa^{2n}}
\bigg).
\eeqn
The above leads to
\[
\frac{\varkappa+\sqrt{m^2-(\omega\pm\Lambda)^2}}{\Lambda(\Lambda\pm2\omega)}
=\frac{2\varkappa}{\Lambda(\Lambda\pm2\omega)}-\frac{1}{2\varkappa}
-
\sum_{n=2}^\infty
\frac{(2n-3)!!
\Lambda^{n-1}(\Lambda\pm 2\omega)^{n-1}}{n!2^n\varkappa^{2n-1}},
\]
and then we derive from \eqref{r-lambda}:
\beqn\nonumber
Q(\Lambda)&=&\frac{\varkappa+\sqrt{m^2-(\omega+\Lambda)^2}}{\Lambda(\Lambda+2\omega)}
+\frac{\varkappa+\sqrt{m^2-(\omega-\Lambda)^2}}{\Lambda(\Lambda-2\omega)}\\
\label{S-exp}
&=&-\frac{1}{\varkappa}-\frac{4\varkappa}{4\omega^2-\Lambda^2}
-\sum_{n=1}^\infty\frac{(2n-1)!!
\Lambda^{n}\big((\Lambda+2\omega)^{n}+(\Lambda- 2\omega)^{n}\big)}{(n+1)!2^{n+1}\varkappa^{2n+1}}<0
\eeqn
since each summand in $\sum\limits_{n=1}^\infty$   is positive for $0<\Lambda<2\omega$. 
This  immediately implies that (\ref{eq1}) has no solution  in the second domain. 

\smallskip
\noindent
{\it (3)}\quad
In the third domain  $m-\omega<2\omega$.  Let us show that $Q(\Lambda)$  decreases monotonically on $(0,m-\omega)$. Indeed, (\ref{S-exp}) implies
\beqn\nonumber
Q'(\Lambda)&=&-\frac{8\Lambda\varkappa}{(4\omega^2-\Lambda^2)^2}-\sum_{n=1}^\infty\frac{(2n-1)!!
n\Lambda^{n-1}\big((\Lambda+2\omega)^{n}+(\Lambda- 2\omega)^{n}\big)}{(n+1)!2^{n+1}\varkappa^{2n+1}}\\
\nonumber
&-&\sum_{n=1}^\infty\frac{(2n-1)!!
n\Lambda^{n}\big((\Lambda+2\omega)^{n-1}+(\Lambda- 2\omega)^{n-1}\big)}{(n+1)!2^{n+1}\varkappa^{2n+1}}<0.
\eeqn
Hence, similarly to (\ref{Qmin}),
$$
\sup\limits_{(0,m-\omega)}|Q(\Lambda)|= |Q(m-\omega)|=\frac{1} {\sqrt{m^2-(2\omega-m)^2}-\varkappa}-\frac{1} {\varkappa}
=\frac{1} {\varkappa}K_{\omega}^{-1},
$$
and 
$$
\kappa\varkappa |Q(\Lambda)|<K_\omega\varkappa |Q(m-\omega)|=1.
$$
Hence (\ref{eq1})
has no solutions
in this domain. 
\end{proof}

Now we collect all the facts  about the eigenvalues  obtained above.

\begin{theorem}[Eigenvalues of $\bfA(\omega,\kappa)$]
\label{theorem-d-lambda}
\quad
\begin{enumerate}
\item
\label{theorem-d-lambda-1}
For $\kappa=\omega^2/m^2$,
one has
$\sigma\sb{\mathrm{p}}(\bfA(\omega,\kappa))=\{0\}$
of algebraic multiplicity $4$;
its geometric multiplicity is $1$ if $\omega\ne 0$
and $2$ if $\omega=0$;
\item
\label{theorem-d-lambda-2}
For $\omega\ne 0$,
$\kappa=0$,
one has $\sigma\sb{\mathrm{p}}(\bfA(\omega,0))=\{0,\,\pm 2\omega\jj\}$
($\pm 2\omega\jj\in\sigma\sb{\mathrm{ess}}(\bfA(\omega,\kappa))$
when $\abs{\omega}\ge m/3$);
\item
\label{theorem-d-lambda-3}
For
$\kappa>\omega^2/m^2$,
one has
$\sigma\sb{\mathrm{p}}(\bfA(\omega,\kappa))=\{0,\,\pm\lambda\}$
with some $\lambda>0$;
\item
\label{theorem-d-lambda-4}
For $\kappa\in(K_\omega,\omega^2/m^2)$,
one has
$\sigma\sb{\mathrm{p}}(\bfA(\omega,\kappa))=\{0,\pm\jj\Lambda\}$,
$\Lambda>0$;
\item
\label{theorem-d-lambda-5}
$\kappa\le K_\omega$, $\kappa\ne 0$,
one has $\sigma\sb{\mathrm{p}}(\bfA(\omega,\kappa))=\{0\}$.
\end{enumerate}
\end{theorem}

\begin{figure}[ht!]
\begin{center}
\includegraphics[width=12cm]{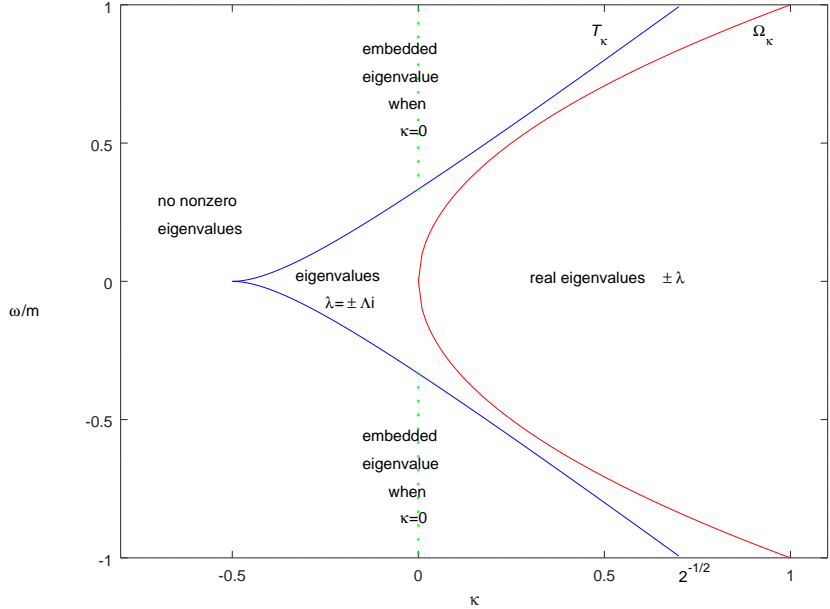}

%\ac{REMEMBER TO UNCOMMENT THE PICTURE!!!}
\end{center}
\caption{
Location of simple eigenvalues $\pm\lambda$
for different values of parameters
$\kappa\in\R$ and $\omega\in(-m,m)$.
The values of $\omega$
on the Kolokolov curve
$\Omega_\kappa$ ($\kappa=\omega^2/m^2$) correspond to
collision of two simple eigenvalues at $\lambda=0$
as indicated by the Kolokolov condition.
The values of $\omega$ on the virtual level curve
$\calT_\kappa$
(\textcolor{blue}{blue curves} for $\kappa\in(-1/2, 1/\sqrt{2}$)
correspond to
virtual levels at thresholds
$\pm\jj(m-\abs{\omega})$.
The regions between the Kolokolov and virtual level lines
(dots on the plot)
correspond to two simple purely imaginary eigenvalues $\pm\jj\Lambda$
in the spectral gap.
The values $\abs{\omega}\ge m/3$ at $\kappa=0$
%(\textcolor{green}{green dots})
correspond to embedded eigenvalues $\pm 2\omega\jj$.
Besides eigenvalues mentioned on this plot,
there is always eigenvalue $\lambda=0$.
}
\label{fig-omegakgsimple}
\end{figure}

\begin{proof}
Part~\itref{theorem-d-lambda-1} in the case $\omega\ne 0$
follows from Lemma~\ref{lemma-sigma-a}~\itref{lemma-sigma-a-5}
and from Lemma \ref{K0} in the case $\omega=0$.
Moreover, in the case $\omega\ne 0$,
one can arrive at the same conclusion 
from considering the function $D_{\omega,\kappa}(\lambda)$.
We have:
\begin{eqnarray}\label{asd}
D_{\omega,\kappa}(\lambda)=(\alpha-2\nu_+)(\alpha-2\nu_-)(1+\kappa\varkappa Q(\jj\lambda)).
\end{eqnarray}
We notice that
$$
\alpha-2\nu_{\pm}=2(\varkappa-\sqrt{\varkappa^2+\lambda\pm2\jj\lambda\omega})=O(\lambda),\quad \lambda\to 0.
$$
Moreover, by (\ref{S-exp}),
\[
Q(\jj\lambda)=-\frac{1}{\varkappa}-\frac{\varkappa}{\omega^2(1+\frac{\lambda^2}{4\omega^2})}+O(\lambda^2)
=-\frac{1}{\varkappa}-\frac{\varkappa}{\omega^2}+O(\lambda^2)
=-\frac{m^2}{\varkappa\omega^2}+O(\lambda^2),\quad |\lambda|<2|\omega|,
\]
therefore,
\[
1+\kappa\varkappa Q(\jj\lambda)
=1-\frac{\kappa m^2}{\omega^2}+O(\lambda^2)
=\left\{\begin{array}{cc}
O(1),\quad  \omega\ne\Omega_\kappa,\\
O(\lambda^2),\quad \omega=\Omega_\kappa,
\end{array}\right.
\qquad\abs{\lambda}\le 2|\omega|.
\]
It follows that
if $\omega=\Omega_\kappa$,
then $\lambda=0$ is a root of \eqref{asd}
of order four.

Part~\itref{theorem-d-lambda-2} follows from  Lemma~\ref{K0},
and
Part~\itref{theorem-d-lambda-3} follows from Lemma~\ref{lemma-sigma-a}~\itref{lemma-sigma-a-3}.
Let us prove
Part~\itref{theorem-d-lambda-4}.   
By Part~\itref{theorem-d-lambda-1},
for $\omega=\Omega_\kappa$, $\sigma\sb{\mathrm{p}}(\bfA(\omega,\kappa))=\{0\}$ of algebraic multiplicity $4$.
For $\omega\gtrless\Omega_\kappa$, the two nonzero eigenvalues
start moving away from the origin,
becoming purely
real for $\omega<\Omega_\kappa$ and 
purely imaginary for $\omega>\Omega_\kappa$.
These two purely imaginary eigenvalues hit the threshold points $\pm\jj(m-\abs{\omega})$
at some values
$\pm\calT_\kappa$,
with $\Omega_\kappa<\calT_\kappa<m$,
becoming virtual levels (see Lemma~\ref{lemma-vl}), since
by Part~\itref{theorem-d-lambda-1},
there are no eigenvalues at this threshold.
It remains to note that  $\omega=\calT_\kappa$ corresponds to
$\kappa=K_{\omega}$. Indeed, solving (\ref{cT}),  we get
$$
\kappa=K_{\omega}=\frac12\Big(\sqrt{\frac{\omega^2}{m^2}+\frac{\omega}{m}}+\frac{\omega}{m}-1\Big)
=\frac{2\Big(\frac{\abs{\omega}}{m+\abs{\omega}}\Big)^{1/2}-1}
{2-2\Big(\frac{\abs{\omega}}{m+\abs{\omega}}\Big)^{1/2}}.
$$
The condition for $\calT_\kappa$ to be smaller than $m$
(so that the root of $D_{\omega,\kappa}(\lambda)$
indeed arrives at $\jj(m-\abs{\omega})$
when $\omega=\pm\calT_\kappa\in(-m,m)$) is
$
(1+2\kappa)^2<3+4\kappa$,
which gives the requirement
$\kappa\in(-\frac{1}{2},\frac{1}{\sqrt{2}})$.

Part~\itref{theorem-d-lambda-5} follows from  Lemma~\ref{lemma-negative}.
The absence of nonzero roots for $\kappa< K_\omega$,
$\kappa\ne 0$, is explained by the fact that
that when we passing the thresholds $\kappa=K_{\omega}$
the nonzero root $\lambda(\omega,\kappa)$ of $D_{\omega,\kappa}(\lambda)$
moves onto the unphysical sheet  of the Riemann surface of $D_{\omega,\kappa}(\lambda)$.
Recall that in the case  $\kappa=0$, $\abs{\omega}\ge m/3$
the point spectrum also contains
embedded eigenvalues $\pm 2\omega\jj$ by
Lemma~\ref{lemma-vl}~\ref{lemma-vl-1}).

This completes the proof of Theorem~\ref{theorem-d-lambda}.
\end{proof}
%%%%%%%%%%%%%%%%%%%%%%%%%%%%%%%%%%%%%%%%%%%%%%%%%%%%%%%%%%%%%%%%%%%%%%%%%%%%%%%%%%%%%%%%%
 %%%%%%%%%%%%%%%%%%%%%%%%%%%%%%%%%%%%%%%%%%%%%%%%%%%%%%%%%%%%%%%%%%%%%%%%%%

\appendix
\section{Appendix: Reducing $D_{\omega,\kappa}(\lambda)=0$ to a cubic with explicit solution}
\label{appendix}

For the completeness,
we provide an explicit solution to
the equation $D_{\omega,\kappa}(\lambda)=0$,
with $D_{\omega,\kappa}(\lambda)$ from \eqref{def-d-lambda},
which we can cast as a cubic equation.
Albeit explicit, this solution does not yield
the spectral properties of $\bfA(\omega,\kappa)$
as readily as the methods employed in Section~\ref{section-4};
yet this solution enables, for example, the analysis of
large $\kappa$ asymptotics
of roots of $D_{\omega,\kappa}(\lambda)$ on physical and unphysical sheets of the
Riemann surface of $D_{\omega,\kappa}(\lambda)$.
Denote $\varSigma=\nu_++\nu_-$,
with $\nu_\pm=\nu_\pm(\omega,\lambda)$ from \eqref{def-nu};
one has
$$
-2\nu_+\nu_-
=-\varSigma^2+(\nu_+^2+\nu_-^2)
=-\varSigma^2+(2m^2-2\omega^2+2\lambda^2)
=-\varSigma^2+2\lambda^2+\alpha^2/2.
$$
Then (\ref{deter-nu}) takes the form
\begin{equation}\label{eqn2}
\alpha^2(1+\kappa)^2-2\alpha(1+\kappa)\varSigma
-2(-\varSigma^2+2\lambda^2+\alpha^2/2)-\alpha^2\kappa^2=0
\end{equation}
and then
$
\varSigma^2-\alpha(1+\kappa)\varSigma+(\alpha^2\kappa-2\lambda^2)=0.
$
Equation (\ref{eqn2}) implies
\begin{equation}\label{eqn4}
\varSigma=
\fra{\big(
\alpha(1+\kappa)\pm\sqrt{\alpha^2(1-\kappa)^2+8\lambda^2}
\,\big)}{2}.
\end{equation}
Further, from (\ref{eqn2}) and (\ref{eqn4}),
\begin{eqnarray*}
-2\nu_+\nu_-
&=&-\varSigma^2+2\lambda^2+\alpha^2/2
=\alpha^2(1+\kappa)^2/2-\alpha^2\kappa^2/2
-\alpha(1+\kappa)\varSigma
\\
&=&\frac{\alpha^2(1+\kappa)^2}{2}-\frac{\alpha^2\kappa^2}{2}-\frac{\alpha(1+\kappa)}2\Big(\alpha(1+\kappa)\pm\sqrt{\alpha^2(1-\kappa)^2+8\lambda^2}\Big)
\\
&=&-\frac{\alpha^2\kappa^2}2\mp\frac{\alpha(1+\kappa)}{2}\sqrt{\alpha^2(1-\kappa)^2+8\lambda^2},
\end{eqnarray*}
so
\begin{eqnarray}\label{KK1}
\nu_+^2\nu_-^2
=\frac{\alpha^4\kappa^4}{16}+\frac{\alpha^2(1+\kappa)^2}{16}
\big(\alpha^2(1-\kappa)^2+8\lambda^2\big)
\pm\frac{\alpha^3(1+\kappa)\kappa^2}{8}\sqrt{\alpha^2(1-\kappa)^2+8\lambda^2}.
\end{eqnarray}
On the other side,  the definition of $\nu_\pm$
\eqref{def-nu}
implies that
\begin{eqnarray}\label{KK2}
\nu_+^2\nu_-^2
&=&
((\omega+\jj \lambda)^2-m^2)((\omega-\jj\lambda)^2-m^2)
=((\omega-m)^2+\lambda^2)((\omega+m)^2+\lambda^2) 
\nonumber
\\
&=&
(\omega^2-m^2)^2+\lambda^4+2\lambda^2(\omega^2+m^2)
=\frac{\alpha^4}{16}+\lambda^4+4\lambda^2m^2-\frac{\lambda^2\alpha^2}2.
\end{eqnarray}
From (\ref{KK1}) and (\ref{KK2}),
simplifying and denoting $x=\lambda^2$, we obtain:
\begin{eqnarray}\label{x-x-x}
x^2+4x m^2-x \alpha^2-\frac{\alpha^4\kappa^4}8+\frac{\alpha^4\kappa^2}8
-\frac{x \alpha^2\kappa^2}2
-x \alpha^2\kappa
=
\pm\frac{\alpha^3(1+\kappa)\kappa^2}{8}\sqrt{\alpha^2(1-\kappa)^2+8x}.
\end{eqnarray}
Defining
\begin{eqnarray}\label{def-c}
c(\omega,\kappa)=4m^2-\alpha^2-\alpha^2\kappa-\fra{\alpha^2\kappa^2}{2}
\end{eqnarray}
and squaring \eqref{x-x-x}
yields
$$
\Big(
x^2+c x+\frac{\alpha^4\kappa^2(1-\kappa^2)}{8}
\Big)^2
=\frac{\alpha^8(1-\kappa^2)^2\kappa^4}{64}+\frac{\alpha^6(1+\kappa)^2\kappa^4 x}8.
$$
There is a root $x=0$.
To find nonzero roots,
we simplify the above and cancel $x$, arriving at
\begin{eqnarray}\label{cubic-0}
x^3+2c x^2+x
\Big[c^2+\frac{\alpha^4\kappa^2(1-\kappa^2)}{4}\Big]+c\frac{\alpha^4\kappa^2(1-\kappa^2)}{4}
-\frac{\alpha^6(1+\kappa)^2\kappa^4}8=0.
\end{eqnarray}
Writing
\begin{eqnarray}\label{x-y}
x=y-\fra{2c}{3},
\end{eqnarray}
we reduce the cubic equation (\ref{cubic-0}) to the form
\begin{eqnarray}\label{cubic}
y^3+p y+q=0
\end{eqnarray}
with
\begin{eqnarray}\label{def-p-q}
p=-\frac{c^2}{3}+\frac{\alpha^4\kappa^2(1-\kappa^2)}{4},
\qquad
q=-\frac{2c^3}{27}+\frac{c\alpha^4\kappa^2(1-\kappa^2)}{12}-\frac{\alpha^6(1+\kappa)^2\kappa^4}8.
\end{eqnarray}
If the discriminant
\begin{equation}\label{def-discriminant}
\Delta(\omega,\kappa)= -4p^3-27q^2
\end{equation}
with $c=c(\omega,\kappa)$ from \eqref{def-c},
is negative,
then there is exactly one real root
$y\in\R$
of equation \eqref{cubic}
(and two complex conjugate roots with nonzero imaginary part).
The real root
is given by
$
y_1=\big(-\frac{q}{2}+\frac{(-\Delta)^{1/2}}{108}\big)^{1/3}
+
\big(-\frac{q}{2}-\frac{(-\Delta)^{1/2}}{108}\big)^{1/3},
$
and then \eqref{x-y} yields
\begin{eqnarray}\label{x1}
x_1=\Big(-\frac{q}{2}+\frac{(-\Delta)^{1/2}}{108}\Big)^{1/3}
+
\Big(-\frac{q}{2}-\frac{(-\Delta)^{1/2}}{108}\Big)^{1/3}-\frac{2c}{3}.
\end{eqnarray}
Let us mention that
roots of \eqref{cubic}
with nonzero imaginary part
correspond to roots $\lambda=\pm\sqrt{y-2c/3}$ of $D_{\omega,\kappa}(\lambda)$
with nonzero real and imaginary parts.
By
Lemma~\ref{lemma-sigma-a}~\itref{lemma-sigma-a-3},
these roots do not correspond to eigenvalues of $\bfA(\omega,\kappa)$.

We are going to show that
the discriminant $\Delta(\omega,\kappa)$
is negative
either if $\abs{\omega}<m$ is sufficiently close to $m$
or if $\abs{\kappa}$ is sufficiently large,
so that there is exactly one real solution to
\eqref{cubic}
and therefore exactly one pair of roots
to $D_{\omega,\kappa}(\lambda)=0$
which are located on the Riemann sheet
corresponding to eigenvalues of $\bfA(\omega,\kappa)$
(that is, when the real parts of
$\nu_{+}$ and $\nu_{-}$ from \eqref{def-nu} are positive).
%%as long as $\kappa>0$ and $\varkappa>0$ are sufficiently small.

\begin{lemma}\label{lemma-delta-negative}
\begin{enumerate}
\item
There is $\kappa_0>0$ such that $\Delta(\omega,\kappa)<0$
for $\abs{\kappa}>\kappa_0$.
\item
For each $\kappa\in\R\setminus\{-1,0\}$,
there is $\omega_\kappa\in(0,m)$ such that
$\Delta(\omega,\kappa)<0$ for $\abs{\omega}\in(\omega_\kappa,m)$.
\end{enumerate}
%% There are $\kappa_0>0$ and $\varkappa_0>0$  such that
%% $\Delta<0$
%% for $\kappa\in(0,\kappa_0)$,
%% $\varkappa=\sqrt{m^2-\omega^2}\in(0,\varkappa_0)$.
\end{lemma}

\begin{proof}
Using \eqref{def-p-q},
one computes:
\begin{eqnarray*}
&&
\hskip -20pt
\Delta(\omega,\kappa)=
-4p^3-27q^2
\\
&&
\hskip -20pt
=
-\frac{\alpha^6\kappa^4(1+\kappa)^2}{2}
\left(
c^3
-\frac{c^2\alpha^2(1-\kappa)^2}{8}
-\frac{9c\alpha^4\kappa^2(1-\kappa^2)}{8}
+\frac{\alpha^6\kappa^2(1-\kappa^2)(1-\kappa)^2}{8}
+\frac{27 \alpha^6(1+\kappa)^2\kappa^4}{32}
\right).
\end{eqnarray*}
Since $c=4m^2-\alpha^2(1+\kappa+\kappa^2/2)$,
one concludes that,
for each fixed $\kappa\in\R$ different from $-1$ and $0$,
$\Delta(\omega,\kappa)<0$
if $\alpha=2\sqrt{m^2-\omega^2}$ is sufficiently small
(that is, if $\abs{\omega}$ is sufficiently close to $m$).

Alternatively,
one can keep the highest order powers of $\kappa$,
substituting $c\sim-\alpha^2\kappa^2/2$ and getting
\[
\Delta(\omega,\kappa)
\sim
-\alpha^6\kappa^6
\Big(
-\frac{1}{2}
\frac{\alpha^6\kappa^6}{8}
-
\frac{\alpha^4\kappa^4}{4}
\frac{\alpha^2\kappa^2}{16}
+
\frac{\alpha^2\kappa^2}{2}
\frac{9\alpha^4\kappa^4}{16}
+
\frac{\alpha^6\kappa^6}{16}
+
\frac{27\alpha^6\kappa^6}{64}
\Big)
%\]
%\[
%% =
%% -\alpha^{12}\kappa^{12}
%% \Big(
%% -\frac{1}{16}-\frac{1}{64}
%% -\frac{9}{32}
%% +\frac{1}{16}
%% +\frac{27}{64}
%% \Big)
=
-\frac{\alpha^{12}\kappa^{12}}{8},
\]
showing that
$\Delta(\omega,\kappa)$
is strictly negative for $\abs\kappa$ large enough.
\end{proof}
%%%%%%%%%%%%%%%%%%%%%%%%%%%%%%%%%%%%%%%%%%%%%%%%%%%%
\begin{remark}
Let us note that
for $\abs{\kappa}\to\infty$,
one has:
\[
c\sim-\frac{\alpha^2\kappa^2}{2},
\qquad
(-\Delta)^{1/2}\sim\frac{\alpha^6\kappa^6}{2\sqrt{2}},
\qquad
q
\sim
-\frac{2}{27}c^3-\frac{\alpha^6\kappa^6}{8}
\sim
-\frac{29}{27}\frac{\alpha^6\kappa^6}{8},
\] 
and \eqref{x1} yields
\[
x_1\sim
\Big(
\frac{\alpha^6\kappa^6}{8}
+
\frac{1}{108}\frac{\alpha^6\kappa^6}{2\sqrt{2}}
\Big)^{1/3}
+
\Big(
\frac{\alpha^6\kappa^6}{8}
-
\frac{1}{108}\frac{\alpha^6\kappa^6}{2\sqrt{2}}
\Big)^{1/3}
+
\frac{2}{3}\frac{\alpha^2\kappa^2}{2}
\sim
c\alpha^2\kappa^2,
\]
with some $c>1$,
showing that
$\lambda_{1,2}=\pm\sqrt{x_1}$ are real. For  $\kappa>0$
these values $\lambda_{1,2}$ are exactly
the two real nonzero eigenvalues of $\bfA(\omega,\kappa)$
which appear in Theorem~\ref{theorem-d-lambda}~\itref{theorem-d-lambda-3}.
For $\kappa<0$, since $\bfA(\omega,\kappa)$ has no
real eigenvalues
(Lemma~\ref{lemma-sigma-a}~\itref{lemma-sigma-a-3}),
$\lambda_{1,2}$ correspond
to resonances of $\bfA(\omega,\kappa)$
(that is, to zeros of \eqref{def-d-lambda}
on one of the unphysical Riemann sheets
of $D_{\omega,\kappa}(\lambda)$).
\end{remark}

The expression \eqref{x1}
can be used
to analyze zero eigenvalues of $\bfA(\omega,\kappa)$.

\begin{lemma}
$x_1=0$ in the following cases:
$\kappa=-1$, $\kappa=0$, and $\kappa=\omega^2/m^2$.
\end{lemma}

\begin{proof}
Since $x_1=0$ corresponds to $y_1=2c/3$ (see \eqref{x-y}),
we need to solve
\[
(2c/3)^3+(2c/3)p+q=0.
\]
Substituting  $c$ from \eqref{def-c},
$p$, $q$ from \eqref{def-p-q}, and simplifying,
one arrives at
\[
\Big(4m^2-(1+\kappa+\kappa^2/2)\alpha^2\Big)
\frac{\alpha^4 \kappa^2(1-\kappa^2)}{4}
-\frac{\alpha^6(1+\kappa)^2\kappa^4}{8}=0.
\]
Canceling common factors
(corresponding to solutions $\kappa=-1$ and $\kappa=0$),
one obtains $\omega^2=\kappa m^2$.
\end{proof}

\begin{remark}
The value $\kappa=\omega^2/m^2$
is in agreement with the value of $\Omega_\kappa$ from
\eqref{def-o-k}.
%%\eqref{the-requirement}.
By Lemma~\ref{lemma-sigma-a},
we know that only the cases
$\kappa=0$ and $\kappa=\omega^2/m^2$
correspond to higher algebraic multiplicity of eigenvalue
$\lambda=0$,
while the case $\kappa=-1$ has to correspond to
resonances:
one can see from \eqref{def-d-lambda}
that in this case
the values $\lambda=0$ are located on the Riemann sheet
of $D_{\omega,\kappa}(\lambda)$
characterized by $\nu\sb\pm<0$.
\end{remark}

\def\cprime{\mbox{'}} \def\polhk#1{\setbox0=\hbox{#1}{\ooalign{\hidewidth
  \lower1.5ex\hbox{`}\hidewidth\crcr\unhbox0}}}

\end{document}